\documentclass[12pt,reqno]{amsart}
\usepackage{amsaddr}
\usepackage{amssymb,latexsym,amsmath,epsfig,amsthm,mathrsfs}
\usepackage{rotating}
\usepackage{graphicx}
\usepackage{amssymb}
\usepackage{lineno}
\usepackage{enumitem}
\usepackage{cite}
\usepackage[top=3cm, bottom=3cm, left=3cm, right=3cm]{geometry}
\usepackage[usenames]{color}
\usepackage[colorlinks=true,
linkcolor=blue,
filecolor=blue,
citecolor=blue]{hyperref}

\newtheorem{thma}{Theorem}

\newtheorem{lemmaa}{Lemma}

\newtheorem{theorem}{Theorem}[section]

\newtheorem{lemma}[theorem]{Lemma}
\newtheorem{conjecture}{Conjecture}

\theoremstyle{definition}

\theoremstyle{remark}

\theoremstyle{definition}

\numberwithin{equation}{section}

\begin{document}
	
	\title[Restricted Signed Sumsets - II]{Direct and Inverse Problems for Restricted Signed Sumsets - II}
	
	
\author[R K Mistri]{Raj Kumar Mistri$^{*}$}
\address{\em{\small Department of Mathematics, Indian Institute of Technology Bhilai\\
Durg, Chhattisgarh, India\\
e-mail: rkmistri@iitbhilai.ac.in}}


\author[N Prajapati]{Nitesh Prajapati$^{**}$}
\address{\em{\small Department of Mathematics, Indian Institute of Technology Bhilai\\
Durg, Chhattisgarh, India\\
email: niteshp@iitbhilai.ac.in}}

\thanks{$^{*}$Corresponding author}
\thanks{$^{**}$The research of the author is supported by the UGC Fellowship (NTA Ref. No.: 211610023414)}

	\subjclass[2010]{Primary 11P70; Secondary 11B13, 11B75}

	\keywords{sumsets, restricted sumsets, signed sumset, restricted signed sumsets.}

\begin{abstract}
Let $A=\{a_{1},\ldots,a_{k}\}$ be a nonempty finite subset of an additive abelian group $G$. For a positive integer $h$, the {\em restricted $h$-fold signed sumset of $A$}, denoted by $h^{\wedge}_{\pm}A$, is defined as
	$$h^{\wedge}_{\pm}A=\left\lbrace \sum_{i=1}^{k} \lambda_{i} a_{i}: \lambda_{i} \in \left\lbrace -1, 0, 1\right\rbrace \ \text{for} \ i=  1, 2, \ldots, k \ \text{and} \  \sum_{i=1}^{k} \left|\lambda_{i} \right| =h\right\rbrace. $$
A direct problem for the restricted $h$-fold signed sumset is to find the optimal size of $h^{\wedge}_{\pm}A$ in terms of $h$ and $|A|$. An inverse problem for this sumset is to determine the structure of the underlying set $A$ when the sumset has optimal size. While the signed sumsets (which is defined differently compared to the restricted signed sumset) in finite abelian groups has been investigated by Bajnok and Matzke, the restricted $h$-fold signed sumset $h^{\wedge}_{\pm}A$ is not well studied even in the additive group of integers $\Bbb Z$. Bhanja, Komatsu and Pandey studied these  problems for the restricted $h$-fold signed sumset for $h=2, 3$, and $k$, and conjectured some direct and inverse results for $h \geq 4$. In a recent paper, Mistri and Prajapati proved these conjectures completely for the set of positive integers. In this paper, we prove these conjectures for the set of nonnegative integers, which settles all the conjectures completely.
	\end{abstract}

	\maketitle
	

\section{Introduction}
Let $\Bbb Z$ denote the set of integers. For  integers $a$ and $b$ with $a \leq b$, let $[a, b]$ denote the interval of integers $\{n \in \Bbb Z: a \leq n \leq b\}$. Let $|S|$ denote the size of the finite set $S$. An {\em arithmetic progression (A.P.)} of integers with $k$-terms and common difference $d$ is a set $A$ of the form $\{a + id: i = 0, 1, \ldots, k-1\}$, where $a, d \in \Bbb Z$ and $d \neq 0$. Let $G$ be an additive abelian group, and let $A = \{a_1, \ldots, a_k\}$ be a nonempty finite subset of $G$. Let $h$ be a positive integer. Then the {\em $h$-fold sumset of the set $A$}, denoted by $hA$ is defined as
\[hA = \left\{\sum_{i=1}^{k}\lambda_i a_i: \lambda_i \in [0, h] ~\text{for}~ i = 1, \ldots, k ~\text{and}~ \sum_{i=1}^{k}\lambda_i = h\right\}.\]
The {\em restricted $h$-fold sumset of the set $A$}, denoted by $h^{\wedge}A$, is defined as
\[h^{\wedge}A =  \left\{\sum_{i=1}^{k}\lambda_i a_i: \lambda_i \in [0, 1] ~\text{for}~ i = 1,\ldots, k ~\text{and}~ \sum_{i=1}^{k}\lambda_i = h \right\}.\]

The study of the sumsets dates back to Cauchy \cite{cauchy} who proved that if $A$ and $B$ are nonempty subsets of $\Bbb Z_p$, then $|A + B| \geq \min(p, |A| + |B| -1)$, where $A + B = \{a + b: a \in A, b \in B\}$ and $\Bbb Z_p$ is the group of prime order $p$. The result has been known as {\em Cauchy-Davenport Theorem} after Davenport rediscovered this result \cite{dav1, dav2} in $1935$. The Cauchy-Davenport theorem immediately implies that if $A$ is a nonempty subsets of $\Bbb Z_p$, then $|hA| \geq \min(p, h|A| - h + 1)$ for all positive integers $h$. The corresponding theorem for the restricted $h$-fold sumset $h^{\wedge}A$ in $\Bbb Z_p$ is due to Dias da Silva and Hamidoune (see \cite{dias}) who proved using the theory of exterior algebra that if $A$ is a nonempty subsets of $\Bbb Z_p$, then $|h^{\wedge}A| \geq \min(p, h|A| - h^2 + 1)$ for all positive integers $h \leq |A|$. Later this result was reproved by Alon, Nathanson and Ruzsa (see \cite{alon1} and \cite{alon2}) by means of {\em polynomial method} which is a very powerful tool for tackling certain problems in additive combinatorics. The theorem for $h = 2$ is known as {\em Erd\H{o}s - Heilbronn conjecture} which was first conjectured by Erd\H{o}s and Heilbronn (see \cite{erdos}) during $1960$.
	
The following theorem provides the optimal lower bound for the size of restricted $h$-fold sumset $h^{\wedge}A$ in the additive group of integers $\Bbb Z$.
\begin{thma}[{\cite[Theorem 1]{nath1}; \cite[Theorem 1.9]{nath}}]\label{restricted-hfold-direct-thm}
Let $A$ be a nonempty finite set of integers with $|A| = k$ and let $1 \leq h \leq k.$ Then
\begin{equation}\label{res-sum-bound}
  |h^{\wedge}A| \geq hk - h^2 + 1.
\end{equation}
The lower bound in $(\ref{res-sum-bound})$ is best possible.
\end{thma}

Next theorem characterizes the sets $A \subseteq \Bbb Z$ for which the equality holds in $(\ref{res-sum-bound})$.
\begin{thma}[{\cite[Theorem 2]{nath1}; \cite[Theorem 1.10]{nath}}]\label{restricted-hfold-inverse-thm}
Let $k \geq 5$ and let $2 \leq h \leq k - 2.$ If $A$ is a set of $k$ integers such that
\[|h^{\wedge}A| = hk - h^2 + 1,\]
then $A$ is a $k$-term arithmetic progression.
\end{thma}

These sumsets and other kind of sumsets have been studied extensively in liteartures (see \cite{nath, tao, freiman, mann} and the references given therein).

Two other variants of these sumsets have appeared recently in the literature in the works of Bajnok, Ruzsa and other researchers \cite{bajnok-ruzsa2003, bajnok-matzke2015,bajnok-matzke2016, bajnok2018, bhanja-pandey2019, bhanja-kom-pandey2021, klopsch-lev2003, klopsch-lev2009}: the {\it $h$-fold signed sumset} $h_{\pm}A$ and the {\it restricted $h$-fold signed sumset} $h^{\wedge}_{\pm}A$ of the set $A$ which are defined as follows:
	\begin{equation*}
		h_{\pm}A = \left\{\sum_{i=1}^{k} \lambda_i a_i : \lambda_i \in [-h, h] ~\text{for}~ i = 1, 2, \ldots ,k ~\text{and}~ \sum_{i=1}^{k} |\lambda_i| =h \right\},
	\end{equation*}
	and
	\begin{equation*}
		h^{\wedge}_{\pm}A = \left\{\sum_{i=1}^{k} \lambda_i a_i : \lambda_i \in [-1, 1] ~\text{for}~ i = 1, 2, \ldots ,k ~\text{and}~ \sum_{i=1}^{k} |\lambda_i| =h \right\}.
	\end{equation*}
It is easy to see that
\[h^{\wedge}A \cup h^{\wedge}(-A) \subseteq h_{\pm}^{\wedge}A \subseteq h_{\pm}^{\wedge}(A \cup -A),\]
and
\[h^{\wedge}_{\pm}A \subseteq h_{\pm}A,\]
For a nonzero integer $c$, we have
\[h^{\wedge}_{\pm}(c \ast A) = c \ast (h^{\wedge}_{\pm}A),\]
where $c \ast A = \{ca: a \in A\}$ and $-A = (-1) \ast A $.

While $h$-fold sumsets are well-studied in the literature, the $h$-fold signed sumsets are not well-studied in the literature. The signed sumsets appear naturally in the literature in many contexts. The $h$-fold signed sumset $h_{\pm}A$ first appeared in the work of Bajnok and Ruzsa \cite{bajnok-ruzsa2003} who studied it in the context of the independence number of a subset of an abelian group $G$ (see also \cite{bajnok2000} and \cite{bajnok2004}), and it also appeared in the work of Klopsch and Lev \cite{klopsch-lev2003,klopsch-lev2009} in the context of dimameter of the group $G$ with respect to the set $A$. 

For a positive integer $m \leq |G|$, define
\[\rho(G, m, h) = \min \{|hA| : A \subseteq G, |A|=m\}\] 
and   
\[\rho_{\pm}(G, m, h) = \min \{|h_{\pm}A| : A \subseteq G, |A|=m\}.\] 
Bajnok and Matzke initiated the detailed study of the function $\rho_{\pm}(G, m, h)$, and they proved that $\rho_{\pm}(G, m, h) = \rho(G, m, h)$, when $G$ is a finite cyclic group \cite{bajnok-matzke2015}. In another work, they studied the cases when $\rho_{\pm}(G, m, h) = \rho(G, m, h)$, where $G$ is an elementary abelian group \cite{bajnok-matzke2016}. In a recent paper \cite{bhanja-pandey2019}, Bhanja and Pandey have studied the direct and inverse problems in the additive group $\mathbb{Z}$ of integers. They obtained the optimal lower bound for the cardinality of the sumset $h_{\pm}A$. They also proved that if the optimal lower bound is achieved, then $A$ must be a certain arithmetic progression.

In case of restricted signed sumset $h^{\wedge}_{\pm}A$, not much is known even in the additive group of integers $\Bbb Z$. The direct problem for the sumset $h_{\pm}^{\wedge}A$ is to find lower bounds for $|h_{\pm}^{\wedge}A|$ in terms of $|A|$. The inverse problem for this sumset is to determine the structure of the finite sets $A$ of for which $|h_{\pm}^{\wedge}A|$ is optimal. In this direction, recently, Bhanja, Komatsu and Pandey \cite{bhanja-kom-pandey2021} solved some cases of both the direct and inverse problems for  $h^{\wedge}_{\pm}A$ in $\Bbb Z$ and conjectured for the rest of the cases. More precisely, they proved the following result.
\begin{thma}[{\cite[Theorem 2.1]{bhanja-kom-pandey2021}}]\label{thm:3}
Let $h$ and $k$ be positive integers with $h \leq k$. Let $A$ be a set of $k$ positive integers. Then
\begin{equation}\label{Bound A}
	\left|h^{\wedge}_{\pm}A\right| \geq   2(hk-h^2)+ \frac{h(h+1)}{2} + 1.
\end{equation}
The lower bound in $\eqref{Bound A}$ is best possible for $h=1, 2$ and $k$.
\end{thma}

	\begin{thma}[{\cite[Theorem 2.3]{bhanja-kom-pandey2021}}]\label{thm:4}
		Let $h \geq 3$ be a positive integer. Let $A$ be the set of $h$ positive integers such that $|h_{\pm}^\wedge A| = \frac{h(h + 1)}{2} + 1$. Then 
		\begin{equation*}
			A =
			\begin{cases}
				\{a_1, a_2, a_1 + a_2\} ~\text{with}~ 0 < a_1 < a_2, &\text{if $h = 3$};\\
				d \ast [1, h] ~\text{for some positive integer}~ d, &\text{if $h \geq 4$.}  
			\end{cases}
		\end{equation*}
	\end{thma}

\begin{thma}[{\cite[Theorem 3.1]{bhanja-kom-pandey2021}}]\label{thm:5}
		Let $h$ and $k$ be integer such that $1 \leq h \leq k$. Let $A$ be set of $k$ nonnegative integers such that $0 \in A$. Then
\begin{equation}\label{Bound B}
  |h_{\pm}^\wedge A| \geq 2(hk - h^2) + \frac{h(h - 1)}{2} + 1.
\end{equation} 
		The lower bound in $\eqref{Bound B}$ is best possible for $h = 1, 2$, and $k$. 
	\end{thma}

\begin{thma}[{\cite[Theorem 3.3]{bhanja-kom-pandey2021}}]\label{thm:6}
		Let $h \geq 4$ be a positive integer. Let $A$ be the set of $h$ nonnegative integers with $0 \in A$ such that $|h_{\pm}^\wedge A| = \frac{h(h - 1)}{2} + 1$. Then 
		\begin{equation}
			A =
			\begin{cases}
				\{0, a_1, a_2, a_1 + a_2\} ~\text{with}~ 0 < a_1 < a_2, &\text{if $h = 4$};\\
				d \ast [0, h - 1] ~\text{for some positive integer}~ d, &\text{if $h \geq 5$.}
			\end{cases}
		\end{equation}
	\end{thma}
	
In the same paper, they proved the inverse theorems for $|2_{\pm}^\wedge A|$ also (see \cite[Theorem 2.2, Theorem 2.3, Theorem 3.2, and Theorem 3.3]{bhanja-kom-pandey2021}). It can be verified that the lower bounds in (\ref{Bound A}) is not optimal for $3 \leq h \leq k-1$. For these cases, they conjectured the lower bounds and the inverse results, and proved these conjectures for the case $h = 3$ (see \cite[Theorem 2.5 and Theorem 3.5]{bhanja-kom-pandey2021})). The precise statements of the conjectures are the following:
	
	\begin{conjecture}[{\cite[Conjecture 2.4, Conjecture 2.6]{bhanja-kom-pandey2021}}]\label{Conjecture 1}			
		Let $A$ be a set of $k \geq 4$ positive integers, and let $h$ be an integer with $3 \leq h \leq k-1$. Then
		\begin{equation}\label{Lower bound Conjecture 1}
			\left|h^{\wedge}_{\pm}A\right| \geq 2hk-h^2 + 1.
		\end{equation}
		This lower bound is best possible.

 Moreover, if the equality holds in $(\ref{Lower bound Conjecture 1})$, then $A = d \ast \{1,3,\ldots, 2k-1\}$ for some positive integer $d$.
	\end{conjecture}
	
	\begin{conjecture}[{\cite[Conjecture 3.4]{bhanja-kom-pandey2021}}]\label{Conjecture 2}
		Let $A$ be a set of $k \geq 5$ nonnegative integers with $0 \in A$, and let $h$ be an integer with $3 \leq h \leq k-1$. Then
		\begin{equation}\label{Lower bound Conjecture 2}
			\left|h^{\wedge}_{\pm}A\right| \geq 2hk - h(h+1) + 1.
		\end{equation}
		This lower bound in $(\ref{Lower bound Conjecture 2})$ is best possible.
	\end{conjecture}

\begin{conjecture}[{\cite[Conjecture 3.7]{bhanja-kom-pandey2021}}]\label{Conjecture 3}
		Let $A$ be a set of $k \geq 5$ nonnegative integers with $0 \in A$, and let $h$ be an integer with $3 \leq h \leq k-1$. If
		\begin{equation}\label{Lower bound Conjecture 3}
			\left|h^{\wedge}_{\pm}A\right| = 2hk - h(h+1) + 1,
		\end{equation}
then $A = d \ast [0,k-1]$ for some positive integer $d$.
\end{conjecture}
	
Mohan, Mistri and Pandey confirmed the conjecture for $h =4$, and they also proved the conjectures for certain special types of sets, including arithmetic progression \cite{mmp2024}. The authors proved Conjecture \ref{Conjecture 1} in an earlier paper \cite{mistri-prajapati2025a}. We remark that Conjecture \ref{Conjecture 3} is not true for $h = 4$ and $k = 5$ which can be seen by taking the set $A = \{0, 1, 2, 4, 6\}$.

In this paper, we prove Conjecture \ref{Conjecture 2}, correct Conjecture \ref{Conjecture 3} and prove it. More precisely we prove the following two theorems. 
	
\begin{theorem}\label{rssn-thm-1}
		Let $h$ and $k$ be positive integers such that $3 \leq h \leq k - 1$ and $k \geq 5$. Let $A = \{a_1, a_2, \ldots, a_k\}$ be a set of nonnegative integers such that $0 = a_1 < a_2 < \cdots < a_k$. Then
		\begin{equation}\label{rssn-thm-1-eq1}
			|h_{\pm}^\wedge A| \geq 2hk - h(h + 1) + 1.
		\end{equation}
		The lower bound in \eqref{rssn-thm-1-eq1} is best possible.
	\end{theorem}
	\begin{theorem}\label{rssn-thm-2}
		Let $h$ and $k$ be positive integers such that $3 \leq h \leq k - 1$ and $k \geq 5$. Let $A = \{a_1, a_2, \ldots, a_k\}$ be a set of nonnegative integers such that $0 = a_1 < a_2 < \cdots < a_k$. If
		\[|h_{\pm}^\wedge A| = 2hk - h(h + 1) + 1,\]
		then 
		\begin{equation*}
			A =
			\begin{cases}
				a_2 \ast [0, 4] ~\text{or}~ a_2 \ast \{0, 1, 2, 4, 6\}, &\text{if $h = 4$, $k = 5$};\\
				a_2 \ast [0, k - 1], &\text{otherwise.}  
			\end{cases}
		\end{equation*}
	\end{theorem}

For a set $A \subseteq \Bbb Z$, let $A_{abs} = \{|a|: a \in A\}$. It is easy to verify that if $A$ is a nonempty finite set of integers such that either $A \cap (-A) = \emptyset$ or $A \cap (-A) = \{0\}$, then
\[h_{\pm}^{\wedge} A = h_{\pm}^{\wedge} A_{abs}.\]

This identity and above theorems immediately imply the following theorems.

\begin{theorem}\label{rssn-thm-1a}
		Let $h$ and $k$ be positive integers such that $3 \leq h \leq k - 1$ and $k \geq 5$. Let $A$ be a set of $k$ integers such that $A \cap (-A) = \{0\}$. Then
		\begin{equation*}\label{rssn-thm-1a-eq1}
			|h_{\pm}^\wedge A| \geq 2hk - h(h + 1) + 1.
		\end{equation*}
		This lower bound is best possible. 
	\end{theorem}
	\begin{theorem}\label{rssn-thm-2a}
		Let $h$ and $k$ be positive integers such that $3 \leq h \leq k - 1$ and $k \geq 5$. Let $A$ be a set of $k$ integers such that $A \cap (- A) = \{0\}$. If
		\[|h_{\pm}^\wedge A| = 2hk - h(h + 1) + 1,\]
		then 
		\begin{equation*}
			A_{abs} =
			\begin{cases}
				d \ast [0, 4] ~\text{or}~ d \ast \{0, 1, 2, 4, 6\}, &\text{if $h = 4$, $k = 5$};\\
				d \ast [0, k - 1], &\text{otherwise,}  
			\end{cases}
		\end{equation*}
where $d$ is the smallest nonzero element of $A_{abs}$
	\end{theorem}

Mohan, Mistri and Pandey proved the following lemma for the restricted signed sumset.
\begin{lemmaa}[{\cite[Lemma $2$]{mmp2024}}]\label{rssn-basic-lem1}
Let $h$ and $k$ be integers such that $3 \leq h \leq k - 1$ and $k \geq 5$. Let $A = \{a_1, a_2, \ldots, a_k\}$ be a set of integers such that $0 = a_1 < a_2 < \cdots < a_k$. Let $B = \{a_1, a_2 \ldots, a_{h + 1}\} \subseteq A$. If $|h_{\pm}^\wedge B| \geq h^2 + h + 1 + t$, where $t \geq 0$, then 
\[|h_{\pm}^\wedge A| \geq 2hk - h^2 - h + 1 + t.\]
\end{lemmaa} 

In view of the above lemma, to prove Theorem \ref{rssn-thm-1}, it suffices to prove the following theorem.
\begin{theorem}\label{rssn-thm-3}
Let $h$ be an integers such that $h \geq 3$. Let $A = \{a_1, a_2, \ldots, a_{h + 1}\}$ be a set of integers such that $0 = a_1 < a_2 < \cdots < a_{h + 1}$. Then
\begin{equation}\label{rssn-thm-3-eq1}
  |h_{\pm}^\wedge A| \geq h^2 + h + 1.
\end{equation}
The lower bound in $\eqref{rssn-thm-3-eq1}$ is best possible.
\end{theorem}

In Section \ref{section-rssn-aux-lemma}, we prove some auxiliary lemmas which will be used to prove Theorem \ref{rssn-thm-3} (hence Theorem \ref{rssn-thm-1}) and Theorem \ref{rssn-thm-2} in Section \ref{section-rrsn-thm-proof}. 

\section{Auxiliary lemmas} \label{section-rssn-aux-lemma}
We follow the following notations from \cite{mistri-prajapati2025a}. For a nonempty finite set $A$ of integers, let $\min(A)$, $\max(A)$, $\min_{+}(A)$, $\max_{-}(A)$ denote the smallest, the largest, the second smallest, and the second largest elements of $A$, respectively.  A set $S$ is said to be {\it symmetric} if $x \in S$ implies $-x \in S$. For a subset $A$ of an additive abelian group $G$, if $c \in G$, then we write $c + A$ for $\{c\} + A$. The {\em set of subsums of $A \subseteq G$}, denoted by $\Sigma (A)$ is defined as follows.  
\[\Sigma (A) = \bigg\{\sum_{b \in B}b : B \subseteq A\bigg\}.\]

We also recall the following facts from \cite{mistri-prajapati2025a} which will be used frequently in the proofs of lemmas.
\begin{enumerate}
	\item Let $h \geq 3$ be an integer. Let $A = \{a_1,\ldots, a_{h + 1}\}$ be a set of integers such that $a_1 < \cdots < a_{h + 1}$, and
		\[a_i \not \equiv a_j\pmod {2}\]
		for some  $i, j \in [1, h + 1]$, where $i \neq j$. Let $A_i = A \setminus \{a_i\}$, and let $A_j = A \setminus \{a_j\}$. Then it is easy to see that the sumsets $h_{\pm}^\wedge A_i$ and $h_{\pm}^\wedge A_j$ are disjoint.

	\item Let $h \geq 2$ be an integer. Let $A = \{a_1,\ldots, a_h\}$ be a set of integers such that $a_1 < \cdots < a_h$. Then it is easy to show that
		\[h_{\pm}^\wedge A = \min (h_{\pm}^\wedge A) + 2 \ast \Sigma (A) = \max (h_{\pm}^\wedge A) - 2 \ast \Sigma (A).\]
\end{enumerate}
	
We need the following lemmas of \cite{mistri-prajapati2025a} to prove Theorem \ref{rssn-thm-2} and Theorem \ref{rssn-thm-3}.

	\begin{lemma}[{\cite[Lemma $2.1$]{mistri-prajapati2025a}}]\label{rssp-lem1}
		Let $h \geq 3$ be an integer. Let $A = \{a_1,\ldots, a_{h + 1}\}$ be a set of positive integers such that $a_1 < \cdots < a_{h + 1}$. Furthermore, assume that
		\[a_1 \equiv a_2\pmod {2} ~\text{and}~ a_r \not \equiv a_1 \pmod {2}\]
		for some  $r \in [3, h + 1]$. Then
		\[|h_{\pm}^\wedge A| \geq |h_{\pm}^\wedge A_r| + \frac{h (h + 1)}{2} + 2h + 1,\]
		where $A_r = A \setminus \{a_r\}$. Hence 
		\[|h_{\pm}^\wedge A| \geq h^2 + 3h + 2.\]	
	\end{lemma}
	
	\begin{lemma}[{\cite[Lemma $2.2$]{mistri-prajapati2025a}}]\label{rssp-lem2}
		Let $h \geq 3$ be an integer. Let $A = \{a_1, a_2,\ldots, a_{h}\}$ be a set of odd positive integers such that $a_1 < \cdots < a_{h}$. Then 
		\begin{equation}\label{rssp-lemeq:13}
			|h_{\pm}^\wedge A|\geq h^2 - 1.
		\end{equation}	
		The lower bound in \eqref{rssp-lemeq:13} is best possible.
	\end{lemma}	
	
	\begin{lemma}[{\cite[Lemma $2.5$]{mistri-prajapati2025a}}]\label{rssp-lem4}
		Let $h \geq 3$ be an integer. Let $A = \{a_1, \ldots, a_{h + 1}\}$ be a set of positive integers such that $a_1 < \cdots < a_{h + 1}$. Furthermore, assume that 
		\[a_1 \not \equiv a_2 \pmod {2} ~\text{and}~ a_3 \not \equiv a_1 \pmod {2}.\]
		Then
		\begin{equation}
			|h_{\pm}^\wedge A| \geq 
			\begin{cases}
				|h_{\pm}^\wedge A_1| + \frac{h(h + 1)}{2} + 2h - 1, &\text{if $a_3 = 2a_1 + a_2$};\\
				|h_{\pm}^\wedge A_1| + \frac{h(h + 1)}{2} + 3h - 2, &\text{if $a_3 \neq 2a_1 + a_2$.}   
			\end{cases}
		\end{equation}
		where $A_1 = A \setminus \{a_1\}$. Hence 
		\begin{equation}
			|h_{\pm}^\wedge A| \geq 
			\begin{cases}
				h^2 + 3h, &\text{if $a_3 = 2a_1 + a_2$};\\
				h^2 + 4h - 1, &~\text{if $a_3 \neq 2a_1 + a_2$.}   
			\end{cases}
		\end{equation}
	\end{lemma} 

	\begin{lemma}[{\cite[Lemma $2.6$]{mistri-prajapati2025a}}]\label{rssp-lem5}
		Let $h \geq 4$ be an integer. Let $A = \{a_1, \ldots, a_{h + 1}\}$ be a set of positive integers such that $a_1 < \cdots < a_{h + 1}$. Furthermore, assume that 
		\[a_2 \not \equiv a_1 \pmod {2} ~\text{and}~ a_3 \equiv a_1 \pmod {2}.\]
		Then
		\[|h_{\pm}^\wedge A| \geq |h_{\pm}^\wedge A_2| + \frac{h(h + 1)}{2} + h,\]
		where $A_2 = A \setminus \{a_2\}$. Hence
		\begin{equation}
			|h_{\pm}^\wedge A| \geq 
			\begin{cases}
				h^2 + 2h + 2, &\text{if $A_2$ is not in an A. P. and $h \geq 4$};\\
				\frac{1}{2} h(3h - 1) + 4, &\text{if $h \geq 4$, $A_2$ is an A. P. and $a_2 \not \equiv 0 \pmod 2$};\\
				26, &\text{if $A_2$ is an A. P. and $a_2 \equiv 0 \pmod 2$, $h = 4$};\\
				2h(h - 1), &\text{if $A_2$ is an A. P. and $a_2 \equiv 0 \pmod 2$, $h \geq 5$.} 
			\end{cases}
		\end{equation}
	\end{lemma} 
	
	We need the following additional lemmas for the proof Theorem \ref{rssn-thm-2} and Theorem \ref{rssn-thm-3}.

	\begin{lemma}\label{rssn-lem8}
		Let $h \geq 4$ be an integer. Let $A = \{a_1, a_2, \ldots, a_{h}\}$ be a set of integers such that $A \setminus \{0\}$ be set of odd integers, and let $0 = a_1 < a_2 < \cdots < a_{h}$. Then 
		\[|h_{\pm}^\wedge A|\geq h^2 - 2h.\]	
	\end{lemma}
	
	\begin{proof}
		It is easy to see that 
		\[h_{\pm}^\wedge A = 0 + (h - 1)_{\pm}^\wedge A_1,\]
		where $A_1 = A \setminus \{0\}$. Hence by applying Lemma \ref{rssp-lem2}, we get
		\[|h_{\pm}^\wedge A| = |(h - 1)_{\pm}^\wedge A_1| \geq (h - 1)^2 - 1 = h^2 - 2h.\]
		This completes the proof.	
	\end{proof}
	
	\begin{lemma}\label{rssn-lem9}
		Let $h \geq 4$ be an integer. Let $A = \{a_1, a_2, \ldots, a_{h + 1}\}$ be a set of nonnegative integers such that $0 = a_1 < a_2 < \cdots < a_{h + 1}$. Furthermore, assume that
		\[a_3 \equiv a_2 \pmod {2} ~\text{and}~ a_r \not \equiv a_2 \pmod {2}\] for some  $r \in [4, h + 1]$. Then
		\[|h_{\pm}^\wedge A| \geq |h_{\pm}^\wedge A_r| + \frac{h(h - 1)}{2} + 2h + 1,\]
		where $A_r = A \setminus \{a_r\}$. Hence 
		\[|h_{\pm}^\wedge A| \geq h^2 + h + 2.\]	
	\end{lemma}
	
	\begin{proof}
		Let $S = \{a_2 + \cdots + a_{h + 1}, - a_2 - \cdots - a_{h + 1}\}$, and let $A_1 = A \setminus \{0\}$. Then 
		\[S \cup (0 + (h - 1)_{\pm}^\wedge A_1) \subseteq h_{\pm}^\wedge A.\]
		Since
		\[- a_2 - \cdots - a_{h + 1} < \min (0 + (h - 1)_{\pm}^\wedge A_1) < \max (0 + (h - 1)_{\pm}^\wedge A_1) < a_2 + \cdots + a_{h + 1},\]
		it follows that 
		\[	|h_{\pm}^\wedge A| \geq |(0 + (h - 1)_{\pm}^\wedge A_1)| + |S|.\]
		Hence
		\begin{equation}\label{rssn-lemeq:3}
			|h_{\pm}^\wedge A| \geq |(h - 1)_{\pm}^\wedge A_1| + 2.
		\end{equation}
		It follows from Lemma \ref{rssp-lem1} that
		\[|(h - 1)_{\pm}^\wedge A_1| \geq |(h - 1)_{\pm}^\wedge B_r| + \frac{h(h - 1)}{2} + 2 (h - 1) + 1,\]
		where $B_r = A \setminus \{0, a_r\}$.
		It is easy to see that 
		\[|(h - 1)_{\pm}^\wedge B_r| = |0 + (h - 1)_{\pm}^\wedge B_r| = |h_{\pm}^\wedge A_r|.\]
		Thus we have
		\begin{align*}
			|(h - 1)_{\pm}^\wedge A_1| & \geq |h_{\pm}^\wedge A_r| + \frac{h(h - 1)}{2} + 2 (h - 1) + 1.
		\end{align*}
		Therefore,
		\begin{equation}\label{rssn-lemeq:4}
			|(h - 1)_{\pm}^\wedge A_1| \geq |h_{\pm}^\wedge A_r| + \frac{h(h - 1)}{2} + 2h - 1.
		\end{equation}
		Using \eqref{rssn-lemeq:3} and \eqref{rssn-lemeq:4}, we get
		\[|h_{\pm}^\wedge A| \geq |h_{\pm}^\wedge A_r| + \frac{h(h - 1)}{2} + 2h + 1.\]
		Hence it follows from Theorem \ref{thm:5} that
		\begin{align*}
			|h_{\pm}^\wedge A| & \geq \frac{h(h - 1)}{2} + 1 + \frac{h(h - 1)}{2} + 2h + 1 = h^2 + h + 2.	
		\end{align*}
		This completes the proof.	
	\end{proof}
	
	\begin{lemma}\label{rssn-lem10}
		Let $h \geq 4$ be an integer. Let $A = \{0, a_2, \ldots, a_{h + 1}\}$ be a set of nonnegative integers such that $0 < a_2 < \cdots < a_{h + 1}$. Furthermore, assume that
		$A \setminus \{0\}$ is the set of odd integers. Then
		\[|h_{\pm}^\wedge A| \geq |h_{\pm}^\wedge A_1| + |h_{\pm}^\wedge A_2|,\]
		where $A_1 = A \setminus \{0\}$ and $A_2 = A \setminus \{a_2\}$. Hence 
		\[|h_{\pm}^\wedge A| \geq 2h(h - 1) - 1.\]	
	\end{lemma}
	
	\begin{proof}
		Since the sumsets $h_{\pm}^\wedge A_1$ and $h_{\pm}^\wedge A_2$ are disjoint subsets of $h_{\pm}^\wedge A$, it follows that 
		\[|h_{\pm}^\wedge A| \geq |h_{\pm}^\wedge A_1| + |h_{\pm}^\wedge A_2|.\]
		Using Lemma \ref{rssp-lem2} and Lemma \ref{rssn-lem8}, we get
		\begin{align*}
			|h_{\pm}^\wedge A| &\geq  (h^2 - 1) + (h^2 - 2h) = 2h(h - 1) - 1.	
		\end{align*}
		This completes the proof.	
	\end{proof}
	
	\begin{lemma}\label{rssn-lem11}
		Let $h \geq 3$ be an integer. Let $A = \{a_1, a_2, \ldots, a_{h + 1}\}$ be a set of nonnegative integers such that $0 = a_1 < a_2 < \cdots < a_{h + 1}$. Furthermore, assume that
		\[a_3 \not \equiv 0 \pmod {2} ~\text{and}~ a_2 \equiv 0 \pmod 2.\]
		Then
		\[|h_{\pm}^\wedge A| \geq |h_{\pm}^\wedge A_3| + \frac{h(h + 1)}{2} + h + 1,\]
		where $A_3 = A \setminus \{a_3\}$. Hence 
		\[|h_{\pm}^\wedge A| \geq h^2 + h + 2.\]  
	\end{lemma}
	
	\begin{proof}
		Let $A_1 = A \setminus \{0\}$. Then
		\[h_{\pm}^\wedge A_3 \cup h_{\pm}^\wedge A_1 \subseteq h_{\pm}^\wedge A.\]
		Now we	consider the subsets $B_0, \ldots, B_h, C_0, \ldots, C_{h - 1}$ of $h_{\pm}^\wedge A_1$ as follows: 
		
		Let $u = - a_2 - \cdots - a_{h + 1}$, and let $v = a_1 - a_3 - \cdots - a_{h + 1}$. Define
		\begin{align*}
			B_0 & = \{u\},\\
			B_1 & = \{u + 2a_i : i = 2, \ldots, h + 1\},\\
			C_0 & = \{v\}.
		\end{align*}	
		For $j = 2, \ldots, h$, define
		\[B_j = \{u + 2a_i + a_{h + 3 - j} + \cdots + a_{h + 1} : i = 2, \dots, h + 2 - j\}.\] 
		Furthermore, for  $j = 1, \ldots, h - 1$, define
		\[C_j = \{v + 2(a_{h + 2 - j} + \cdots + a_{h + 1})\}.\]
		It is easy to see that 	
		\[\max (B_i) < \min (C_i) = \max (C_i) < \min (B_{i + 1})\]
		for $i = 0, \ldots, h - 1$. Hence sets $B_i$ and $C_j$ are disjoint for $i = 0, \ldots, h$ and $j = 0, \ldots, h - 1$. Since the sumsets $h_{\pm}^\wedge A_3$ and $h_{\pm}^\wedge A_1$ are disjoint, it follows that $B_0 \cup \cdots \cup B_h \cup C_0 \cup \cdots \cup C_{h - 1}$ and $h_{\pm}^\wedge A_3$ are disjoint sets. Hence
		\[h_{\pm}^\wedge A \supseteq h_{\pm}^\wedge A_3 \cup B_0 \cup \cdots \cup B_h \cup C_0 \cup \cdots \cup C_{h - 1}.\]
		So, we get
		\begin{align*}
			|h_{\pm}^\wedge A| &\geq |h_{\pm}^\wedge A_3| + \sum_{j = 0}^{h} |B_j| + \sum_{j = 0}^{h - 2} |C_j |\\
			&= |h_{\pm}^\wedge A_3| + 1 + \sum_{j = 1}^{h}|B_j| + \sum_{j = 0}^{h - 1} 1\\
			&= |h_{\pm}^\wedge A_3| + 1 + \sum_{j = 1}^{h}(h - j + 1) + h\\
			&= |h_{\pm}^\wedge A_3| + \frac{h(h + 1)}{2} + h + 1.
		\end{align*}
		Therefore, it follows from Theorem \ref{thm:5} that
		\begin{align*}
			|h_{\pm}^\wedge A| &\geq \frac{h(h - 1)}{2} + 1 + \frac{h(h + 1)}{2} + h + 1 \\
			&= h^2 + h + 2. 
		\end{align*}
		This completes the proof.
	\end{proof}
	
	\begin{lemma}\label{rssn-lem12}
		Let $h \geq 4$ be an integer. Let $A = \{0, a_2, \ldots, a_{h + 1}\}$ be a set of nonnegative integers such that $0 < a_2 < \cdots < a_{h + 1}$. Furthermore, assume that
		\[a_2 \not\equiv 0 \pmod 2,~ a_3 \equiv 0 \pmod 2 ~\text{and}~ a_3 \neq 2a_2.\]
		Then
		\[|h_{\pm}^\wedge A| \geq |h_{\pm}^\wedge A_2| + \frac{h(h + 1)}{2} + h,\]
		where $A_2 = A \setminus \{a_2\}$. Hence
		\begin{equation}\label{rssn-lem12e1}
			|h_{\pm}^\wedge A| \geq 
			\begin{cases}
				23, &\text{if $h = 4$};\\
				h^2 + h + 2, &\text{if $A_2$ is not in an A. P. and $h \geq 5$};\\
				\frac{3}{2} h(h - 1) + 3, & \text{if $A_2$ is an A. P. and $h \geq 5$.} 
			\end{cases}
		\end{equation} 
	\end{lemma}
	
	\begin{proof}
		Let $A_1 = A \setminus \{0\}$. Then
		\[h_{\pm}^\wedge A_1 \cup h_{\pm}^\wedge A_2 \subseteq h_{\pm}^\wedge A.\]
		Now we consider the following cases.
		
		\noindent {\textbf{Case 1}} ($2a_2 < a_3$).
		Define the subsets $B_0, \ldots, B_h, C_0, \ldots, C_{h - 2}$ of $h_{\pm}^\wedge A_1$ as follows: 
		
		Let 
		\[u = - a_2 - \cdots - a_{h + 1} ~\text{and}~ v = 0 + a_2 - a_4 - \cdots - a_{h + 1}.\] 
		Define
		\begin{align*}
			B_0 & = \{u\},\\
			B_1 & = \{u + 2a_i : i = 3, \ldots, h + 1\},\\
			B_h & = \{- u\},\\
			C_0 & = \{v, u + 2a_2\}.
		\end{align*}	
		For $j = 2, \ldots, h - 1$, define
		\[B_j = \{u + 2a_i + a_{h + 3 - j} + \cdots + a_{h + 1} : i = 3, \dots, h + 2 - j\}.\] 
		Furthermore, for  $j = 1, \ldots, h - 2$, define
		\[C_j = 2(a_{h + 2 - j} + \cdots + a_{h + 1}) + C_0.\]
		Clearly, 	
		\[\max (B_i) < \min (C_i) = \max (C_i) < \min (B_{i + 1}),\]
		for $i = 0, \ldots, h - 2$, and
		\[\max (B_{h - 1}) < \min (B_h).\]
		Hence sets $B_i$ and $C_j$ are disjoint for $i = 0, \ldots, h$ and $j = 0, \ldots, h - 2$. Since the sumsets $h_{\pm}^\wedge A_1$ and $h_{\pm}^\wedge A_2$ are disjoint, it follows that $B_0 \cup \cdots \cup B_h \cup C_0 \cup \cdots \cup C_{h - 2}$ and $h_{\pm}^\wedge A_2$ are disjoint sets. Hence
		\[h_{\pm}^\wedge A \supseteq h_{\pm}^\wedge A_2 \cup B_0 \cup \cdots \cup B_h \cup C_0 \cup \cdots \cup C_{h - 2}.\]
		So, we get
		\begin{align*}
			|h_{\pm}^\wedge A| &\geq |h_{\pm}^\wedge A_2| + \sum_{j = 0}^{h} |B_j| + \sum_{j = 0}^{h - 2} |C_j|\\
			&= |h_{\pm}^\wedge A_2| + 2 + \sum_{j = 1}^{h - 1}|B_j| + \sum_{j = 0}^{h - 2} 2\\
			&= |h_{\pm}^\wedge A_2| + 2 + \sum_{j = 1}^{h}(h - j) + 2(h - 1)\\
			&= |h_{\pm}^\wedge A_2| + \frac{h(h - 1)}{2} + 2h\\
			&= |h_{\pm}^\wedge A_2| + \frac{h(h + 1)}{2} + h.
		\end{align*}
		
		\noindent {\textbf{Case 2}} ($a_3 < 2a_2$).
		Now we consider the subsets $B_0, \ldots, B_h, C_0, \ldots, C_{h - 2}$ of $h_{\pm}^\wedge A_1$ as follows: 
		Let $u = - a_2 - \cdots - a_{h + 1}$, and let $v = 0 - a_2 - a_4 - \cdots - a_{h + 1}$. Define
		\begin{align*}
			B_0 & = \{u\},\\
			B_1 & = \{u + 2a_i : i = 2, \ldots, h + 1\},\\
			C_0 & = \{v\}.
		\end{align*}	
		For $j = 2, \ldots, h$, define
		\[B_j =  \{u + 2a_i + a_{h + 3 - j} + \cdots + a_{h + 1} : i = 2, \dots, h + 2 - j\}.\] 
		Furthermore, for  $j = 1, \ldots, h - 2$, define
		\[C_j = \{v + 2(a_{h + 2 - j} + \cdots + a_{h + 1})\}.\]
		Clearly, 	
		\[\max (B_i) < \min (C_i) = \max (C_i) < \min (B_{i + 1}),\]
		for $i = 0, \ldots, h - 2$, and
		\[\max B_{h - 1} < \min B_h.\]
		Hence sets $B_i$ and $C_j$ are disjoint for $i = 0, \ldots, h$ and $j = 0, \ldots, h - 2$. Since the sumsets $h_{\pm}^\wedge A_1$ and $h_{\pm}^\wedge A_2$ are disjoint, it follows that $B_0 \cup \cdots \cup B_h \cup C_0 \cup \cdots \cup C_{h - 2}$ and $h_{\pm}^\wedge A_2$ are disjoint sets. Hence
		\[h_{\pm}^\wedge A \supseteq h_{\pm}^\wedge A_2 \cup B_0 \cup \cdots \cup B_h \cup C_0 \cup \cdots \cup C_{h - 2}.\]
		So, we get
		\begin{align*}
			|h_{\pm}^\wedge A| &\geq |h_{\pm}^\wedge A_2| + \sum_{j = 0}^{h} |B_j| + \sum_{j = 0}^{h - 2} |C_j|\\
			&= |h_{\pm}^\wedge A_2| + 1 + \sum_{j = 1}^{h}|B_j| + \sum_{j = 0}^{h - 2} 1\\
			&= |h_{\pm}^\wedge A_2| + 1 + \sum_{j = 1}^{h}(h - j + 1) + (h - 1)\\
			&= |h_{\pm}^\wedge A_2| + \frac{h(h + 1)}{2} + h.
		\end{align*}
		
		Therefore, we have 
		\begin{equation}\label{rssn-lemeq:16}
			|h_{\pm}^\wedge A| \geq |h_{\pm}^\wedge A_2| + \frac{h(h + 1)}{2} + h.
		\end{equation}

		Now we consider the following cases.
		
		\noindent {\textbf{Case 1}} ($h = 4$). Let 
		\[v = 0 - a_2 - a_4 - a_5.\]
		Then
		\[v \not \equiv x \pmod 2,\]
		where $x \in 4_{\pm}^\wedge A_2$. Hence
		\[v \notin 4_{\pm}^\wedge A_2.\]
		Clearly,
		\[- a_2 - a_3 - a_4 - a_5 < v, a_2 - a_3 - a_4 - a_5 < - a_2 + a_3 - a_4 - a_5.\]
		Since $a_3 \neq 2a_2$, it follows that
		\[v \neq a_2 - a_3 - a_4 - a_5.\]
		Hence
		\[v \notin 4_{\pm}^\wedge A_1.\]
		Therefore,
		\[v \notin 4_{\pm}^\wedge A_1 \cup 4_{\pm}^\wedge A_2.\]
		Since $4_{\pm}^\wedge A_1 \cup h_{\pm}^\wedge A_2$ is a symmetric set and $v \notin 4_{\pm}^\wedge A_1 \cup 4_{\pm}^\wedge A_2$, it follows that
		\[- v \notin 4_{\pm}^\wedge A_1 \cup 4_{\pm}^\wedge A_2.\]
		Since $- v, v \in 4_{\pm}^\wedge A \setminus (4_{\pm}^\wedge A_1 \cup 4_{\pm}^\wedge A_2)$, it follows from \eqref{rssn-lemeq:16} and Theorem \ref{thm:5} that
		\begin{align*}
			|4_{\pm}^\wedge A| &\geq |4_{\pm}^\wedge A_1 \cup 4_{\pm}^\wedge A_2| + 2\\
			&= |4_{\pm}^\wedge A_1| +  |4_{\pm}^\wedge A_2| + 2\\
			&\geq  14 + 7 + 2\\
            &=23.
		\end{align*}
This establishes the first inequality in $\eqref{rssn-lem12e1}$.
		
		\noindent {\textbf{Case 2}} ($h \geq 5$).	
		If $A_2$ is not in an arithmetic progression, then it follows from Theorem \ref{thm:6} that
		\begin{align*}
			|h_{\pm}^\wedge A| &\geq \frac{h(h - 1)}{2} + 2 + \frac{h(h + 1)}{2} + h\\
			&= h^2 + h + 2. 
		\end{align*}
This establishes the second inequality in $\eqref{rssn-lem12e1}$.
		
Now assume that $A_2$ is an arithmetic progression. Since 
		\[a_3 \equiv 0 \pmod 2,\] 
		it follows that
		\[a_3 \equiv \cdots \equiv a_{h + 1} \equiv 0 \pmod 2.\]
		Clearly, 
		\[|h_{\pm}^\wedge A| \geq |h_{\pm}^\wedge A_1| + |h_{\pm}^\wedge A_2|.\]
		Therefore, it follows from Case $1$ in the proof of Lemma \ref{rssp-lem5} and Theorem \ref{thm:5} that
		\begin{align*}
			|h_{\pm}^\wedge A| \geq (h^2 - h + 2) + \frac{h(h - 1)}{2} + 1 = \frac{3}{2}h(h - 1) + 3. 
		\end{align*}
This establishes the last inequality in $\eqref{rssn-lem12e1}$ and thus completes the proof.
	\end{proof}
	
	\begin{lemma}\label{rssn-lem13}
		Let $h \geq 3$ be an integer. Let $A = \{0, a_2, \ldots, a_{h + 1}\}$ be a set of nonnegative integers such that $0 < a_2 < \cdots < a_{h + 1}$. Furthermore, assume that
		\[a_2 \not \equiv 0 \pmod {2}, ~a_3 = 2a_2, ~ a_4 \equiv 0 \pmod 2 ~\text{and}~ a_4 \neq 4a_2.\]
		Then
		\[|h_{\pm}^\wedge A| \geq |h_{\pm}^\wedge A_1| + 2,\]
		where $A_1 = A \setminus \{0\}$. Hence
		\[|h_{\pm}^\wedge A| \geq h^2 + 2h - 2.\]
	\end{lemma}
	
	\begin{proof}
		It follows from \eqref{rssn-lemeq:3} that
		\[|h_{\pm}^\wedge A| \geq |(h - 1)_{\pm}^\wedge A_1| + 2.\]
		By applying Lemma \ref{rssp-lem4}, we get
		\begin{align*}
			|h_{\pm}^\wedge A| \geq (h - 1)^2 + 4(h - 1) - 1 + 2 = h^2 + 2h - 2. 
		\end{align*}
		This completes the proof.	 
	\end{proof}
	
	\begin{lemma}\label{rssn-lem14}
		Let $h \geq 4$ be an integer. Let $A = \{0, a_2, \ldots, a_{h + 1}\}$ be a set of nonnegative integers such that $0 < a_2 < \cdots < a_{h + 1}$. Furthermore, assume that
		\[a_2 \not \equiv 0 \pmod {2}, ~a_3 = 2a_2, ~\text{and}~ a_4 = 4a_2.\]
		Then the following conditions hold.
		\begin{enumerate}
			\item  If $h = 4$ and $A = a_2 \ast \{0, 1, 2, 4, 6\}$, then
			\[|4_{\pm}^\wedge A| = 21.\]
			\item If $h = 4$ and $A \neq a_2 \ast \{0, 1, 2, 4, 6\}$, then
			\[|4_{\pm}^\wedge A| \geq 23.\]
			\item  If $h \geq 5$, then
			\begin{equation*}
				|h_{\pm}^\wedge A| \geq 
				\begin{cases}
					\frac{3}{2} h(h - 1) + 3, &\text{if $A = a_2 \ast \{0, 1, 2j: j = 1, \ldots, h - 1\}$};\\
					h^2 + h + 2, &\text{otherwise}.  
				\end{cases}
			\end{equation*} 
		\end{enumerate}
		
	\end{lemma}
	
	\begin{proof}
		Let $A_1 = A \setminus \{0\}$, and let $A_2 = A \setminus \{a_2\}$. Then 
		\[h_{\pm}^\wedge A_1 \cup h_{\pm}^\wedge A_2 \subseteq h_{\pm}^\wedge A.\]
		Clearly,
		\[h_{\pm}^\wedge A_1 = - (a_2 + a_3 + \cdots + a_{h + 1}) + 2 \ast \Sigma (A_1),\]
		and
		\[h_{\pm}^\wedge A_2 = - (0 + a_3 + \cdots + a_{h + 1}) + 2 \ast \Sigma (A_2).\]
		Hence,
		\[|h_{\pm}^\wedge A_1| = |\Sigma (A_1)| ~\text{and}~ |h_{\pm}^\wedge A_2| = |\Sigma (A_2)|.\]
		Now we consider the subsets $B_{1, 1}, \ldots, B_{h, 1}$ of $\Sigma (A_1)$ as follows:
		\[B_{1, 1} = \{0, a_i : i = 2, \ldots , h + 1\}.\]
		For $j = 2, \ldots, h$, we define
		\[B_{j, 1} = \{a_i : i = 2, \ldots, h + 2 - j\} + a_{h + 3 - j} + \cdots + a_{h + 1}.\]
		Next, we consider the subsets $B_{1, 2}, \ldots, B_{h - 1, 2}$ of $\Sigma (A_2)$ as follows:
		\[B_{1, 2} = \{0, a_i : i = 3, \ldots , h + 1\}.\]
		For $j = 2, \ldots, h - 1$, we define
		\[B_{j, 2} = \{a_i : i = 3, \ldots, h + 2 - j\} + a_{h + 3 - j} + \cdots + a_{h + 1}.\]
		Clearly,
		\[\max (B_{h - 1, 1}) < \min (B_{h , 1}).\]
		and
		\[\max (B_{j, i}) < \min (B_{j + 1, i}),\]
		for $j = 1, \ldots , h - 1$, and $i = 1, 2$. It follows that sets $B_{1, 1}, \ldots, B_{h, 1}$ are disjoint subsets of $\Sigma (A_1)$, and $B_{1, 2}, \ldots, B_{h - 1, 2}$ are disjoint subsets of $\Sigma (A_2)$. Now we consider the following case.
		
		\noindent {\textbf{Case 1}} ($h = 4$).
		Clearly, 
		\[a_3 < a_2 + a_3 = 3a_2 < a_4,\]
		and
		\[a_5 + a_3 < a_5 + a_3 + a_2 < a_5 + a_4.\]
		Now we define
		\[B_{5, 1} = \{ a_2 + a_3,  a_5 + a_3 + a_2\}.\]
		Thus we have $B_{1, 1}, \ldots, B_{5, 1}$ are disjoint subsets of $\Sigma (A_1)$. Observe the following:
		\begin{enumerate}
			\item If $a_5 = 6a_2$, then $A = a_1 \ast \{0, 1, 2, 4, 6\}$, and so
			\[|4_{\pm}^\wedge A| = 21.\]
			\item If $a_5 = 5a_2$, then $A = a_1 \ast \{0, 1, 2, 4, 5\}$, and so
			\[|4_{\pm}^\wedge A| = 23.\]
			\item If $a_5 \neq 5a_2, 6a_2$, then
			\[a_4 < a_5 \neq a_4 + a_2, a_4 + a_3 \neq a_5 + a_2 < a_5 + a_3 ~\text{in}~ \Sigma (A_1),\]
			and
			\[a_4 < a_5 \neq a_4 + a_3 < a_5 + a_3 ~\text{in}~ \Sigma (A_2).\]
\end{enumerate}
Thus we get $2$ extra element in $\Sigma (A_1)$ distinct from the elemets of $B_{1, 1}, \ldots, B_{5, 1}$ and $1$ extra element in $\Sigma (A_2)$ distinct from the elements of $B_{1, 2}, B_{2, 2}, B_{3, 2}$. Hence
			\begin{align*}
				|4_{\pm}^\wedge A| &\geq \sum_{j = 1}^{5} |B_{j, 1}| + |\{a_4 + a_2, a_4 + a_3\}| + \sum_{j = 1}^{3} |B_{j, 2}| + |\{a_4 + a_3\}|\\
				&= (5 + 3 + 2 + 1 + 2) + 2 + (4 + 2 + 1) + 1\\
				&= 23. 
			\end{align*}
			Thus
			\begin{equation*}
				|4_{\pm}^\wedge A| \geq 
				\begin{cases}
					21, &\text{if $A = a_2 \ast \{0, 1, 2, 4, 6\}$};\\
					23, &\text{otherwise.} 
				\end{cases}
			\end{equation*}
This proves parts $(1)$ and $(2)$ of the lemma.		
		
		\noindent {\textbf{Case 2}} $(h \geq 5)$.
		Now we consider the following subcase.
		
		\noindent {\textbf{Subcase 1:}}	($A_2$ is an arithmetic progression).
		In this case, we have
		\[A = a_2 \ast \{0, 1, 2, 4, \ldots, 2h - 2\}.\]
		In this case, it is easy to see that
		\[\Sigma (A_1) = a_2 \ast [0, h^2 - h + 1],\]
		and so
		\begin{align*}
			|h_{\pm}^\wedge A_1| = |\Sigma (A_1)|= h^2 - h + 2. 
		\end{align*}
		Therefore, we have
		\begin{align*}
			|h_{\pm}^\wedge A| &\geq |h_{\pm}^\wedge A_1| + |h_{\pm}^\wedge A_2|\\
			&= |\Sigma (A_1)| + |\Sigma (A_2)|\\
			&\geq (h^2 - h + 2) + \sum_{j = 1}^{h - 1} |B_{j, 2}|\\
			&= (h^2 - h + 2) + h + \sum_{j = 2}^{h - 1} |B_{j, 2}|\\
			&= h^2  + 2 + \sum_{j = 2}^{h - 1}(h - j)\\
			&= \frac{3}{2}h(h - 1) + 3.
		\end{align*}
		
		\noindent {\textbf{Subcase 2}}	($A_2$ is not an arithmetic progression). Suppose that $A_2$ is not in an arithmetic progression. Observe the following:
		\begin{enumerate}
			\item Since $h \geq 5$ and $A_2$ is not in an arithmetic progression, it follows that
			\[|h_{\pm}^\wedge A_2| \geq \frac{h(h - 1)}{2} + 2.\]
			\item It easy to see that
			\[a_{h + 1} + \cdots + a_j + a_3 < a_{h + 1} + \cdots + a_j + a_3 + a_2 < a_{h + 1} + \cdots + a_j + a_4,\]
			for $j = 5 \ldots , h + 1$.	Now we define
			\[B_{h + 1, 1} = \{a_2+ a_3, a_{h + 1} + \cdots + a_j + a_3 + a_2 : j = 5 \ldots, h + 1\}.\] 
			Thus we have $B_{1, 1}, \ldots, B_{h, 1}$ and $B_{h + 1, 1}$ are disjoint subsets of $\Sigma (A_1)$.
			\item  if $a_5 \neq a_4 + a_2$, then we have
\begin{align*}
a_{h + 1} + \cdots + a_6 + a_4 < a_{h + 1} + \cdots + a_6 + a_4 + a_2 &\neq a_{h + 1} + \cdots + a_5 \\
& < a_{h + 1} + \cdots + a_5 + a_2.
\end{align*}
Thus $a_{h + 1} + \cdots + a_6 + a_4 + a_2$ is an element of $\Sigma (A_1)$ but it is different element form the elements of $B_{1, 1} \cup \cdots \cup B_{h + 1, 1}$. Therefore, we have
			\begin{align*}
				|h_{\pm}^\wedge A| &\geq |h_{\pm}^\wedge (A_1)| + |h_{\pm}^\wedge A_2|\\
				&\geq |\Sigma (A_1)| + \frac{h(h - 1)}{2} + 2\\
				&\geq \sum_{j = 1}^{h + 1} |B_{(j , 1)}| + 1 + \frac{h(h - 1)}{2} + 2\\
				&= (h + 1) + \sum_{j = 2}^{h} (h + 1 - j) + (h - 2)  + \frac{h(h - 1)}{2} + h + 3\\
				&= (h + 1) +  \frac{h(h - 1)}{2} + \frac{h(h - 1)}{2} + h + 1\\
				&= h^2 + h + 2.
			\end{align*}
			\item Suppose that $a_5 = a_4 + a_2$ and $a_6 = a_5 + a_2$. If $h = 5$, then we have
			\[A = a_2 \ast \{0, 1, 2, 4, 5, 6\}.\]
			It is easy to see that
			\[|h_{\pm}^\wedge \{0, 1, 2, 4, 5, 6\}| \geq 32\]
			Since $|h_{\pm}^\wedge A| = |h_{\pm}^\wedge \{0, 1, 2, 4, 5, 6\}|$ and $|h_{\pm}^\wedge \{0, 1, 2, 4, 5, 6\}| \geq 32$, it follows that
			\[|h_{\pm}^\wedge A| \geq 32 = h^2 + h + 2.\]
			Now we assume that $h \geq 6$. Let 
			\[y = a_{h + 1} + \cdots + a_7.\]
			Then
			\[y < y + a_6 < y + a_5 + a_3 < y + a_6 + a_3,\]
			and 
			\[y + a_6 + a_4 < y + a_6 + a_5 < y + a_6 + a_4 + a_3 < y + a_6 + a_5 + a_3,\]
			We define
			\[B_{h, 2} = \{a_{h + 1} + \cdots + a_7 + a_5 + a_3, a_{h + 1} + \cdots + a_6 + a_4 + a_3\}.\] 
			It is easy to verify that
			\[(B_{1, 2} \cup \cdots \cup B_{h - 1, 2}) \cap B_{h, 2} = \emptyset\]
			Thus we have $B_{1, 2}, \ldots, B_{h - 1, 2}$ and $B_{h, 2}$ are disjoint subsets of $\Sigma (A_2)$. Therefore, we have
			\begin{align*}
				|h_{\pm}^\wedge A| &\geq |h_{\pm}^\wedge A_1| + |h_{\pm}^\wedge A_2|\\
				&= |\Sigma (A_1)| + |\Sigma (A_2)|\\
				&\geq \sum_{j = 1}^{h + 1} |B_{j, 1}| + \sum_{j = 1}^{h} |B_{j, 2}|\\
				&= (h + 1) + \sum_{j = 2}^{h} |B_{j, 1}|+ (h - 2) + h + \sum_{j = 2}^{h - 1} |B_{j, 2}| + 2\\
				&= \sum_{j = 2}^{h} (h + 1 - j)+ \sum_{j = 2}^{h - 1} (h - j) + 3h + 1\\
				&= h^2 + h + 2. 
			\end{align*}
			\item Suppose that $a_5 = a_4 + a_2$ and $a_6 \neq a_5 + a_2$. If $h = 5$, then it is easy to see that 
			\[a_5 < a_5 + a_2 \neq a_6 < a_6 + a_2,\]
			and
			\[a_5 + a_2 \in \Sigma (A_1).\]
			Clearly 
			\[a_5 + a_2 \notin B_{1, 1} \cup \cdots \cup B_{h + 1, 1}.\]
			 If $h \geq 6$, then we have
			\[y < y + a_5 + a_2 \neq y + a_6 < y + a_6 + a_2,\]
			and 
			\[y + a_5 + a_2 \in \Sigma (A_1),\]
			where $y = a_{h + 1} + \cdots + a_7$. Clearly, 
			\[y + a_5 + a_2 \notin B_{1, 1} \cup \cdots \cup B_{h + 1, 1}.\]
			So, we get one extra element in $\Sigma (A_1)$, which is different element form the elements of $B_{1, 1} \cup \cdots \cup B_{h + 1, 1}$. Therefore, we have
			\begin{align*}
				|h_{\pm}^\wedge A| &\geq |h_{\pm}^\wedge A_1| + |h_{\pm}^\wedge A_2|\\
				&\geq |\Sigma (A_1)| + \frac{h(h - 1)}{2} + 2\\
				&\geq \sum_{j = 1}^{h + 1} |B_{j , 1}| + 1 + \frac{h(h - 1)}{2} + 2\\
				&= (h + 1) + \sum_{j = 2}^{h} (h + 1 - j) + (h - 2)  + \frac{h(h - 1)}{2} + h + 3\\
				&= h^2 + h + 2.
			\end{align*}	
		\end{enumerate}
		From the above observations, it follows that if $A_2$ is not an arithmetic progression, then
		\begin{align*}
			|h_{\pm}^\wedge A| &\geq h^2 + h + 2.
		\end{align*}		
		Therefore, we have
		\begin{equation*}
			|h_{\pm}^\wedge A| \geq 
			\begin{cases}
				\frac{3}{2} h(h - 1) + 3, &\text{if $A = a_2 \ast \{0, 1, 2, 4, \ldots, 2h - 2\}$};\\
				h^2 + h + 2, &\text{otherwise}.  
			\end{cases}
		\end{equation*} 		
		This completes the proof. 	
	\end{proof}
	
	\begin{lemma}\label{rssn-lem20}
		Let $h \geq 4$ be an integer. Let $A = \{0, a_2, \ldots, a_{h + 1}\}$ be a set of nonnegative integers such that $0 < a_2 < \cdots < a_{h + 1}$. Furthermore, assume that
		\[a_2 \not \equiv 0 \pmod {2}, ~a_3 = 2a_2, ~a_4 \not \equiv 0 \pmod {2} ~\text{and}~ a_4 \neq 3a_2.\]
		Then
		\begin{equation*}
			|h_{\pm}^\wedge A| \geq 
			\begin{cases}
				h^2 + h + 2, &\text{if $3a_2 < a_4$};\\
				h^2 + 2h - 2, &\text{if $a_4 < 3a_2$.}  
			\end{cases}
		\end{equation*} 
		Hence	
		\[|h_{\pm}^\wedge A| \geq h^2 + h + 2.\]
	\end{lemma}
	
	\begin{proof}
		Let 
		\[A_1 = A \setminus \{0\} ~\text{and}~ A_2 = A \setminus \{a_2\}.\]
		Let 
		\[x = 0 - a_3 - a_4 - \cdots - a_{h + 1},\]
		\[y = 0 + a_3 - a_4 - \cdots - a_{h + 1},\]
		and 
		\[z = 0 - a_3 + a_4 - \cdots - a_{h + 1}.\]
		Then 
		\[x < y < z.\]
		Now we consider the following case.
		
		\noindent {\textbf{Case 1}} ($a_3 < a_2 + a_3 = 3a_2 < a_4$). Now we consider the subsets $B_1, \cdots, B_h$ of $\Sigma (A_1)$ as follows:
		\begin{align*}
			B_1 &= \{0, a_2 + a_3, a_i : i = 2, \ldots , h + 1\},\\
			B_{h - 1} &= \{a_{h + 1} + \cdots + a_4 + a_2, a_{h + 1} + \cdots + a_4 + a_3\},\\
			B_h &= \{a_{h + 1} + \cdots + a_4 + a_3 + a_2\}.
		\end{align*}
		For $j = 2, \ldots, h - 2$, we define
		\[B_j = \{a_2 + a_3, a_i : i = 2, \ldots, h + 2 - j\} + a_{h + 3 - j} + \cdots + a_{h + 1}.\]
		Clearly,
		\[\max (B_{h - 1}) = a_{h + 1} + \cdots + a_4 + a_3 < a_{h + 1} + \cdots + a_4 + a_3 + a_2 = \min (B_h).\]
		and
		\[\max (B_j) < \min (B_{j + 1}),\]
		for $j = 1, \ldots , h - 1$. It follows that sets $B_1, \ldots, B_h$ are pairwise disjoint subsets of $\Sigma (A_1)$. Hence
		\begin{align*}
			|h_{\pm}^\wedge A_1| &= |\Sigma (A_1)|\\
			&\geq |B_1 \cup \cdots \cup B_h|\\
			&= |B_1| + \cdots + |B_h|\\
			&= |B_1| + \sum_{j = 2}^{h - 2}|B_j|+ |B_{h - 1}| + |B_h|\\
			&= (h + 2) + \sum_{j = 2}^{h - 2} (h + 2 - j) + 2 + 1\\
			&= \frac{h(h + 1)}{2} + h - 1.
		\end{align*}
		Let 
		\[v = 0 + a_2 - a_3 - a_5 - \cdots - a_{h + 1}.\]
		Then
		\[v \not \equiv x \pmod2,\]
		where $x \in h_{\pm}^\wedge A_1$. 
		Hence 
		\[v \notin  h_{\pm}^\wedge A_1.\]
		Clearly, the first three smallest elements of $h_{\pm}^\wedge A_2$ are $x$, $y$, $z$, respectively. Since $a_2 + a_3  = 3a_2 < a_4$, it follows that
		\[y < v < z.\]
		Hence
		\[v \notin h_{\pm}^\wedge A_2.\]
		Since $v \notin h_{\pm}^\wedge A_1$ and $v \notin h_{\pm}^\wedge A_2$, it follows that
		\[v \notin h_{\pm}^\wedge A_1 \cup h_{\pm}^\wedge A_2.\]
		Since $h_{\pm}^\wedge A_1 \cup h_{\pm}^\wedge A_2$ is a symmetric set and $v \notin h_{\pm}^\wedge A_1 \cup h_{\pm}^\wedge A_2$, it follows that 
		\[- v, v \notin h_{\pm}^\wedge A_1 \cup h_{\pm}^\wedge A_2.\]
		Since $-v, v \in h_{\pm}^\wedge A \setminus (h_{\pm}^\wedge A_1 \cup h_{\pm}^\wedge A_2)$, it follows from Theorem \ref{thm:5} that
		\begin{align*}
			|h_{\pm}^\wedge A| &\geq |h_{\pm}^\wedge A_1 \cup h_{\pm}^\wedge A_2| + 2\\
			&\geq |h_{\pm}^\wedge A_1| + |h_{\pm}^\wedge A_2| + 2\\
			&\geq \biggl(\frac{h(h + 1)}{2} + (h - 1) \biggl) + \biggl(\frac{h(h - 1)}{2} + 1 \biggl) + 2\\
			&= h^2 + h + 2.
		\end{align*}	
		
		\noindent {\textbf{Case 2}} ($a_4 < a_2 + a_3 = 3a_2$). Let
		\[u = 0 - a_2 - a_3 - a_5 - \cdots - a_{h + 1}\]
		and
		\[v = 0 + a_2 - a_3 - a_5 - \cdots - a_{h + 1}.\]	
		Then
		\[u = x + a_4 - a_2 ~\text{and}~ v = x + a_4 + a_2.\]
		Since $a_4 < a_2 + a_3 = 3a_2$, it follows that
		\[x < u < v < y.\]
		Now we consider the subsets $B_1, \ldots, B_{h - 1}$ of $h_{\pm}^\wedge A$ as follows:
		\begin{align*}
			B_1 &= x + \{0, a_4 - a_2, a_4 + a_2, 2a_i : i = 3, \ldots , h + 1\},\\
			B_{h - 1} &= \{- x\}.
		\end{align*}
		For $j = 2, \ldots, h - 2$, we define
		\[B_j = x + \{a_4 - a_2, a_4 + a_2, 2a_i : i = 3, \ldots, h + 2 - j\} + 2(a_{h + 3 - j} + \cdots + a_{h + 1}).\]
		Clearly,
		\[\max (B_j) < \min (B_{j + 1})\]
		for $j = 1, \ldots , h - 2$. It follows that sets $B_1, \ldots, B_{h - 1}$ are pairwise disjoint subsets of $h_{\pm}^\wedge A$. Let 
		\[S = B_1 \cup \cdots \cup B_{h - 1}.\]
		Then
		\begin{align*}
			|S| &= |B_1 \cup \cdots \cup B_{h - 1}|\\
			&= |B_1| + \cdots + |B_{h - 1}|\\
			&= |B_1| + \sum_{j = 2}^{h - 2}|B_j|+ |B_{h - 1}|\\
			&= (h + 2) + \sum_{j = 2}^{h - 2} (h + 2 - j) + 1\\
			&= \frac{h(h + 1)}{2} + h - 3.
		\end{align*}
		Clearly, $y \not \equiv z \pmod 2$, where $y \in S$ and $z \in h_{\pm}^\wedge A_1$. Hence the sets $S$ and $h_{\pm}^\wedge A_1$ are disjoint subsets of $h_{\pm}^\wedge A$. By applying  Theorem \ref{thm:3}, we have
		\begin{align*}
			|h_{\pm}^\wedge A| &\geq |S\cup h_{\pm}^\wedge A_1|\\
			&= |S| + |h_{\pm}^\wedge A_1|\\
			&= \biggl(\frac{h(h + 1)}{2} + h - 3\biggl) + \frac{h(h + 1)}{2} + 1\\
			&= h^2 + 2h - 2.
		\end{align*}
		Therefore,
		\begin{equation*}
			|h_{\pm}^\wedge A| \geq 
			\begin{cases}
				h^2 + h + 2, &~\text{if $3a_2 < a_4$};\\
				h^2 + 2h - 2, &~\text{if $a_4 < 3a_2$.}  
			\end{cases}
		\end{equation*} 
		Hence	
		\[|h_{\pm}^\wedge A| \geq h^2 + h + 2.\]
	\end{proof}
	
	\begin{lemma}\label{rssn-lem15}
		Let $h = 4$, and let $A = \{0, a_2, a_3, a_4, a_5\}$ be a set of nonnegative integers. Furthermore, assume that
		\[a_2 \not \equiv 0 \pmod {2}, ~a_3 = 2a_2, ~\text{and}~ a_4 = 3a_2.\]
		Then
		\begin{equation}
			|h_{\pm}^\wedge A| =
			\begin{cases}
				21, &\text{if $A = a_2 \ast \{0, 1, 2, 3, 4\}$};\\
				23, &\text{if $A = a_2 \ast \{0, 1, 2, 3, 5\}$};\\
				25, &\text{if $A = a_2 \ast \{0, 1, 2, 3, 5\}$.} 
			\end{cases}
		\end{equation} 	
		Otherwise,
		\[|h_{\pm}^\wedge A| \geq 22.\]
	\end{lemma}
	
	\begin{proof}
		It is easy to see that 
		\[|4_{\pm}^\wedge \{0, 1, 2, 3, 4\}| = 21,\]
		\[|4_{\pm}^\wedge \{0, 1, 2, 3, 5\}| = 23,\]
		and 
		\[|4_{\pm}^\wedge \{0, 1, 2, 3, 6\}| = 25.\]
		Since $|h_{\pm}^\wedge (d \ast A)| = |h_{\pm}^\wedge A|$, for each nonzero integer $d$, it follows that 
		\[|4_{\pm}^\wedge (a_2 \ast \{0, 1, 2, 3, 4\})| = 21,\]
		\[|4_{\pm}^\wedge (a_2 \ast \{0, 1, 2, 3, 5\})| = 23,\]
		and 
		\[|4_{\pm}^\wedge (a_2 \ast \{0, 1, 2, 3, 6\})| = 25.\]
		Now we assume that 
		\[a_5 \notin \{4a_2, 5a_2, 6a_2\},\]
		\[C_1 = \{a_4 + a_2, a_4 + a_3, a_4 + a_3 + a_2\},\]
		and
		\[C_2 = \{5a_2\} = \{a_4 + a_3\}.\]
		Let $B_{1, 1}, \ldots, B_{4, 1}$ and $B_{1, 2}, B_{2, 2}, B_{3, 2}$ are same define as Lemma \ref{rssn-lem14}. Clearly,
		\[a_4 < \min (C_1) < \max (C_1) < a_5 + a_4 = \max (B_{2, 1}),\]
		and
		\[a_4 < \min (C_2) = \max (C_2) < a_5 + a_4 = \max (B_{2, 2}).\]
		Since $a_5 \notin \{4a_2, 5a_2, 6a_2\}$, it follows that
		\[a_5 \notin \{a_4 + a_2, a_4 + a_3\}, \]
		and
		\[a_4 + a_2 < a_5  + a_2 \neq  a_4 + a_3 < a_5 + a_3.\]
		So,
		\[C_1 \cap (B_{1, 1} \cup \cdots \cup B_{4, 1}) = \emptyset\]
		and
		\[C_2 \cap (B_{1, 2} \cup B_{2, 2} \cup B_{3, 2}) = \emptyset.\]
		Since $C_1 \subseteq \Sigma (A_1)$ and $C_2 \subseteq \Sigma (A_2)$, it follows that 
		\[|\Sigma (A_1)| \geq |B_{1, 1}| + \cdots + |B_{4, 1}| + |C_1| = 14,\]
		and 
		\[|\Sigma (A_2)| \geq |B_{1, 2}| + |B_{2, 2}| + |B_{3, 2}| + |C_2| = 8.\]
		Hence,
		\begin{align*}
			|h_{\pm}^\wedge A| \geq |\Sigma (A_1)| + |\Sigma (A_2)| = 14 + 8 = 22.
		\end{align*}
		This completes the proof.
	\end{proof}
	
We fix the following notations which will be used now onwards.
\subsection*{Notation}
		Let $h \geq 5$ be an integer. Let $A = \{a_1, a_2, \ldots, a_{h + 1}\}$ be a set of nonnegative integers such that $0 = a_1 < a_2 < \cdots < a_{h + 1}$. Furthermore, assume that
		\[a_2 \not \equiv 0 \pmod {2}.\]
		Let $A_1 = A \setminus \{0\}$ and $A_2 = A \setminus \{a_2\}$.
		\begin{enumerate}
			\item We define the subsets $B_{1, 1}, \ldots, B_{h, 1}$ of $\Sigma (A_1)$ as follows:
			\[B_{1, 1} = \{0, a_i : i = 2, \ldots , h + 1\}.\]
			For $j = 2, \ldots, h$, we define
			\[B_{j, 1} = \{a_i : i = 2, \ldots, h + 2 - j\} + a_{h + 3 - j} + \cdots + a_{h + 1}.\]
			\item We consider the subsets $B_{1, 2}, \ldots, B_{h - 1, 2}$ of $\Sigma (A_2)$ as follows:
			\[B_{1, 2} = \{0, a_i : i = 3, \ldots , h + 1\}.\]
			For $j = 2, \ldots, h - 1$, we define
			\[B_{j, 2} = \{a_i : i = 3, \ldots, h + 2 - j\} + a_{h + 3 - j} + \cdots + a_{h + 1}.\]
			\item Let $t$ be an integer such that $0 \leq t \leq h - 4$. For $0 \leq t \leq h - 5$, we denote $C_1^{(t)}$ the set of elements in $\Sigma (A_1)$ such that if $x \in C_1^{(t)}$, then we have either 
			\[a_{h - t} < x < a_{h + 1 - t}\]
			or
			\[{\max}_{-} (B_{t + 1, 1}) < x < {\max}_{-} (B_{t + 2, 1}) ~\text{with}~ x \notin B_{t + 1, 1} \cup B_{t + 2, 1}\]
			and if $x \in C_1^{(h - 4)}$, then we have either 
			\[a_{4} < x < a_{5}\]
			or
			\[{\max}_{-} (B_{h - 3, 1}) < x < \max (B_{h - 2, 1}) ~\text{with}~ x \notin B_{h - 3, 1} \cup B_{h - 2, 1}.\]
			\item Let $t$ be an integer such that $0 \leq t \leq h - 4$. For $0 \leq t \leq h - 5$, we denote $C_2^{(t)}$ the set of elements in $\Sigma (A_2)$ such that if $x \in C_2^{(t)}$, then we have either 
			\[a_{h - t} < x < a_{h + 1 - t}\]
			or
			\[{\max}_{-} (B_{t + 1, 2}) < x < {\max}_{-} (B_{t + 2, 2}) ~\text{with}~ x \notin B_{t + 1, 2} \cup B_{t + 2, 2}\]
			and if $x \in C_2^{(h - 4)}$, then we have either 
			\[a_{4} < x < a_{5}\]
			or
			\[{\max}_{-} (B_{h - 3, 2}) < x < \max (B_{h - 2, 2}) ~\text{with}~ x \notin B_{h - 3, 2} \cup B_{h - 2, 2}.\]
			\item For $t = 0, \ldots, h - 4$, let
			\[\alpha_t = |C_1^{(t)}| + |C_2^{(t)}|\]	
		\end{enumerate}
	
	Following observations will be quite useful in the proofs of forthcoming lemmas.
\begin{enumerate}
		\item $h_{\pm}^\wedge A_1 \cap h_{\pm}^\wedge A_2 = \emptyset$.
		\item Since $h_{\pm}^\wedge A_1$ and $h_{\pm}^\wedge A_2$ are disjoint subsets of $h_{\pm}^\wedge A$, it follows that
		\begin{align*}
			|h_{\pm}^\wedge A| \geq |h_{\pm}^\wedge A_1| + |h_{\pm}^\wedge A_2|= |\Sigma (A_1)| + |\Sigma (A_2)|.
		\end{align*}
		\item Clearly,
		\[\max (B_{h - 1, 1}) < \min (B_{h , 1}).\]
		and
		\[\max (B_{j, i}) < \min (B_{j + 1, i}),\]
		for $j = 1, \ldots , h - 1$, and $i = 1, 2$. It follows that sets $B_{1, 1}, \ldots, B_{h, 1}$ are disjoint subsets of $\Sigma (A_1)$, and $B_{1, 2}, \ldots, B_{h - 1, 2}$ are disjoint subsets of $\Sigma (A_2)$.
		\item  It is easy to verify that 
		\[C_1^{(t)} \cap (B_{1, 1}, \cup \cdots \cup B_{h, 1}) = \emptyset\]
		for $0 \leq t \leq h - 4$.
		\item Similarly, we have 
		\[C_2^{(t)} \cap (B_{1, 2}, \cup \cdots \cup B_{h - 1, 2}) = \emptyset\]
		for $0 \leq t \leq h - 4$.
		\item By definition of $C_1^{(t)}$, we can easily verify that 
		\[C_1^{(i)} \cap C_1^{(j)} = \emptyset,\]
		where $i, j \in [0, h - 4]$ and $i \neq j$. Therefore, sets $C_1^{(0)}, \ldots, C_1^{(h - 4)}$ are disjoint subsets of $\Sigma (A_1)$. So, $B_{i, 1}$ and $C_1^{(t)}$ are disjoint of $\Sigma (A_1)$ for $i = 1, \ldots, h$ and $t = 0, \ldots, h - 4$.
		\item By definition of $C_2^{(t)}$, we can easily verify that 
		\[C_2^{(i)} \cap C_2^{(j)} = \emptyset,\]
		where $i, j \in [0, h - 4]$ and $i \neq j$. Therefore, sets $C_2^{(0)}, \ldots, C_2^{(h - 4)}$ are disjoint subsets of $\Sigma (A_2)$. So, $B_{i, 2}$ and $C_2^{(t)}$ are disjoint of $\Sigma (A_2)$ for $i = 1, \ldots, h - 1$ and $t = 0, \ldots, h - 4$.
		\item Clearly, we have
		\[|B_{1, 1}| + \cdots + |B_{h, 1}| = \frac{h(h + 1)}{2} + 1,\]
and 
		\[|B_{1, 2}| + \cdots + |B_{h - 1, 2}| = \frac{h(h - 1)}{2} + 1.\]
		\item Since $|h_{\pm}^\wedge A| \geq |\Sigma (A_1)| + |\Sigma (A_2)|$, it follows that
\begin{align*}
			|h_{\pm}^\wedge A| &\geq |\Sigma (A_1)| + |\Sigma (A_2)|\\
			&\geq |B_{1, 1} \cup \cdots \cup B_{h, 1} \cup C_1^{(0)} \cup \cdots \cup C_1^{(h - 4)}|\\ 
            & ~~~~~~~~~~~~~~~~~~~~~~~~~~~~~~~~~~~~~~~~~~~~~~~~~~\ + |B_{1, 2} \cup \cdots \cup B_{h - 1, 2} \cup C_2^{(0)} \cup \cdots \cup C_2^{(h - 4)}|\\
			&= \biggl (\sum_{j = 1}^h |B_{j, 1}| + \sum_{j = 0}^{h - 4}|C_1^{(j)}| \biggl) + \biggl (\sum_{j = 1}^{h - 1} |B_{j, 2}| + \sum_{j = 0}^{h - 4}|C_2^{(j)}| \biggl)\\
			&= \biggl (\frac{h (h + 1)}{2} + 1 + \sum_{j = 0}^{h - 4}|C_1^{(j)}| \biggl) + \biggl (\frac{h (h - 1)}{2} + 1 + \sum_{j = 0}^{h - 4}|C_2^{(j)}| \biggl)\\
			&= \sum_{j = 0}^{h - 4}|C_1^{(j)}| + \sum_{j = 0}^{h - 4}|C_2^{(j)}| + h^2 + 2\\
			&= \sum_{t = 0}^{h - 4}(|C_1^{(t)}| + |C_2^{(t)}|) + h^2 + 2\\
			&= \sum_{t = 0}^{h - 4} \alpha_t + h^2 + 2.
		\end{align*}
\end{enumerate}

	\begin{lemma}\label{rssn-lem16}
		Let $h \geq 5$ be an integer. Let $A = \{0, a_2, \ldots, a_{h + 1}\}$ be a set of nonnegative integers such that $0 < a_2 \cdots < a_{h + 1}$. Furthermore, assume that
		\[a_2 \not \equiv 0 \pmod {2}, ~a_3 = 2a_2, ~\text{and}~ a_4 = 3a_2.\]
		Then
		\begin{enumerate}
			\item \begin{equation}\label{rssn-lemeq:1}
				\alpha_t \geq 1
			\end{equation} 
			for $t = 0, 1, \ldots, h - 4$.
			\item If $a_{h + 1 - t} - a_{h - t} = a_3$ for some $t \in [0, h - 4]$, then we have
			\begin{equation}\label{rssn-lemeq:2}
				\alpha_t \geq 
				\begin{cases}
					2, &\text{if $t = 0, h - 4$};\\
					3, &\text{otherwise}.  
				\end{cases}
			\end{equation} 
			\item If $a_{h + 1 - t} - a_{h - t} = a_3$ and $a_{h - t} - a_{h - 1 - t} = a_2$ for some $t \in [0, h - 5]$, then we have
			\begin{equation}\label{rssn-lemeq:5}
				\alpha_t \geq 
				\begin{cases}
					3, &\text{if $t = 0$};\\
					4, &\text{otherwise}.  
				\end{cases}
			\end{equation} 
			\item If $a_{h + 1 - t} - a_{h - t} = a_4$ for $t \in [0, h - 4]$, then we have
			\begin{equation}\label{rssn-lemeq:6}
				\alpha_t \geq 
				\begin{cases}
					3, &\text{if $t = 0$};\\
					6, &\text{otherwise}.  
				\end{cases}
			\end{equation} 
			\item If $a_{h + 1 - t} - a_{h - t} \notin \{a_2, a_3, a_4\}$ for $t = 0, \ldots, h - 4$, then we have
			\begin{equation}\label{rssn-lemeq:7}
				\alpha_t \geq 
				\begin{cases}
					3, &\text{if $t = h - 4$};\\
					5, &\text{otherwise}.  
				\end{cases}
			\end{equation} 
			\item If $a_5 - a_4 \notin \{a_2, a_3, a_4\}$ and $a_6 - a_5 = a_2$, then we have
			\begin{equation}\label{rssn-lemeq:8}
				\alpha_{h - 4} + \alpha_{h - 5} \geq 5.
			\end{equation}
			\item If $a_{h + 1} - a_h = a_4$ and $a_{h} - a_{h - 1} = a_2$, then 
			\begin{equation}\label{rssn-lemeq:9}
				\alpha_0 \geq 4.
			\end{equation}
		\end{enumerate}
	\end{lemma}
	
	\begin{proof}
		We consider the following case.
		
		\noindent {\textbf{Case 1.}} ($t = 0$). In this case, we have
		\begin{enumerate}
			\item If $a_{h + 1} - a_h = a_2$, then we have $a_{h + 1} + a_2 = a_h + a_3$. Since
			\[a_h + a_3 \in \Sigma (A_2),\]
			\begin{align*}
			    {\max}_{-} (B_{1, 2}) &= a_h < a_{h + 1} = \max (B_{1, 2}) < a_h + a_3 < a_{h + 1} + a_3 \\
			    &= \min (B_{2, 2}) < a_{h + 1} + a_{h - 1} = {\max}_{-} (B_{2, 2}),
			\end{align*}
			and 
			\[a_h + a_3 \notin B_{1, 2} \cup B_{2, 2},\]
			it follows that 
			\[a_h + a_3 \in C_2^{(0)}.\] 
			Hence $|C_2^{(0)}| \geq 1$. Therefore, we have
			\[\alpha_0 \geq |C_1^{(0)}| + |C_2^{(0)}| \geq 1.\]
			This proves \eqref{rssn-lemeq:1} for $t = 0$.
			\item If $a_{h + 1} = a_h + a_3$, then we have
			\[a_{h + 1} + a_2 = a_h + a_3 + a_2 = a_h + a_4.\] 
			Since $a_h + a_2 \in \Sigma (A_1)$, and since 
            \begin{align*}
            	{\max}_{-} (B_{1, 1}) &= a_h < a_h + a_2 < a_h + a_3 = a_{h + 1} = \max (B_{1, 1})\\
            	& < a_{h + 1} + a_{h - 1} = {\max}_{-} (B_{2, 1}),
            \end{align*}
			and 
			\[a_h + a_2 \notin B_{1, 1} \cup B_{2, 1},\]
			it follows that 
			\[a_h + a_2 \in C_1^{(0)}.\]
			Hence $|C_1^{(0)}| \geq 1$. Since 
			\[a_h + a_4 \in \Sigma (A_2),\]
			\begin{align*}
				{\max}_{-} (B_{1, 2}) &= a_h < a_{h + 1} = \max (B_{1, 2}) < a_h + a_4 = a_{h + 1} + a_2 \\ 
				&< a_{h + 1} + a_3 = \min (B_{2, 2}) < a_{h + 1} + a_{h - 1} = {\max}_{-} (B_{2, 2}),
			\end{align*}
			and 
			\[a_h + a_4 \notin B_{1, 2} \cup B_{2, 2},\]
			it follows that 
			\[a_h + a_4 \in C_2^{(0)}.\] 
			Hence $|C_2^{(0)}| \geq 1$. Therefore, we have
			\[\alpha_{0} = |C_1^{(0)}| + |C_2^{(0)}| \geq 1 + 1 = 2.\]
			This proves \eqref{rssn-lemeq:2} for $t = 0$. 
			
			\item If $a_{h + 1} - a_h = a_3$ and  $a_h - a_{h - 1} = a_2$, then as proved earlier $a_h + a_2 \in C_1^{(0)}$. So, $|C_1^{(0)}| \geq 1$. Also
			\begin{align*}
				{\max}_{-} (B_{1, 2}) &= a_h = a_{h - 1} + a_2 < a_{h - 1} + a_3 < a_h + a_3 \\
				&= a_{h + 1} = \max (B_{1, 2}) < a_{h + 1} + a_{h - 1} = {\max}_{-} (B_{2, 2}),
			\end{align*}
			it follows that 
			\[a_{h - 1} + a_3 \in C_2^{(0)}.\]
			It was already shown that $a_h + a_4 \in C_2^{(0)}$. Since $a_{h - 1} + a_3 \neq a_h + a_4$, we have $|C_2^{(0)}| \geq 2$. Therefore, 
			\[\alpha_{0} = |C_1^{(0)}| + |C_2^{(0)}| \geq 1 + 2 = 3.\]
			
			This proves \eqref{rssn-lemeq:5} for $t = 0$.
			\item If $a_{h + 1} - a_h = a_4$, then we have $a_{h + 1} = a_h + a_2 + a_3$. Since
			\[a_h + a_2, a_h + a_3 \in \Sigma (A_1),\] 
			\begin{align*}
				{\max}_{-} (B_{1, 1}) &= a_h < a_h + a_2 < a_h + a_3 < a_{h + 1} = \max (B_{1, 1})\\
				& < a_{h + 1} + a_{h - 1} = {\max}_{-} (B_{2, 1}),
			\end{align*}
			and 
			\[a_h + a_2, a_h + a_3 \notin B_{1, 1} \cup B_{2, 1},\]
			it follows that 
			\[a_h + a_2, a_h + a_3 \in C_1^{(0)}.\]
			Hence $|C_1^{(0)}| \geq 2$. Since
			\[a_h + a_3 \in \Sigma (A_2),\]
			\begin{align*}
				{\max}_{-} (B_{1, 2}) &= a_h < a_h + a_3 < a_h + a_4 = a_{h + 1} = \max (B_{1, 2})\\ 
				&< a_{h + 1} + a_3 = \min (B_{2, 2}) < a_{h + 1} + a_{h - 1} = {\max}_{-} (B_{2, 2}),
			\end{align*}
			and 
			\[a_h + a_3 \notin B_{1, 2} \cup B_{2, 2},\]
			it follows that 
			\[a_h + a_3 \in C_2^{(0)}.\] 
			Hence $|C_2^{(0)}| \geq 1$. Therefore, we have
			\[\alpha_{0} = |C_1^{(0)}| + |C_2^{(0)}| \geq 2 + 1 = 3.\]
			This proves \eqref{rssn-lemeq:6} for $t = 0$.
			\item If $a_{h + 1} - a_h = a_4$ and  $a_h - a_{h - 1} = a_2$, then as we proved earlier $a_h + a_2, a_h + a_3 \in C_1^{(0)}$. So, $|C_1^{(0)}| \geq 2$. Also
			\begin{align*}
				{\max}_{-}(B_{1, 2}) &= a_h < a_{h - 1} + a_3 < a_h + a_3 < a_{h + 1} = \max (B_{1, 2}) \\
				&< a_{h + 1} + a_{h - 1} = {\max}_{-} (B_{2, 2}),
			\end{align*}
			it follows that 
			\[a_{h - 1} + a_3 \in C_2^{(0)}.\]
			It was already shown that $a_h + a_3 \in C_2^{(0)}$. Since $a_{h - 1} + a_3 \neq a_h + a_3$, we have $|C_2^{(0)}| \geq 2$. Therefore,  
			\[\alpha_{0} = |C_1^{(0)}| + |C_2^{(0)}| \geq 2 + 2 = 4.\]
			This proves \eqref{rssn-lemeq:9}.
			
			\item Now we assume that $a_{h + 1} - a_h \notin \{a_2, a_3, a_4\}$. It is easy to see that sets $\{a_{h + 1}, a_{h + 1} + a_2, a_{h + 1} + a_3\}$ and $\{a_h + a_2,a_h + a_3, a_h + a_4\}$ are disjoint. Since
			\[a_h + a_2, a_h + a_3, a_h + a_4 \in \Sigma (A_1),\] 
			\begin{align*}
				{\max}_{-} (B_{1, 1}) &= a_h < a_h + a_2 < a_h + a_3 < a_h + a_4 <\\
				& a_{h + 1} + a_4 \leq a_{h + 1} + a_{h - 1} = {\max}_{-} (B_{2, 1}),
			\end{align*}
			and
			\[\{a_h + a_2, a_h + a_3, a_h + a_4\} \cap (B_{1, 1} \cup B_{2, 1}) = \emptyset,\]
			it follows that 
			\[a_h + a_2, a_h + a_3, a_h + a_4 \in C_1^{(0)}.\]
			Hence $|C_1^{(0)}| \geq 3$. Since
			\[a_h + a_3, a_h + a_4 \in \Sigma (A_2),\]
			\begin{align*}
				{\max}_{-} (B_{1, 2}) &= a_h < a_h + a_3 < a_h + a_4 < a_{h + 1} + a_4\\
				& \leq a_{h + 1} + a_{h - 1} = {\max}_{-} (B_{2, 2}),
			\end{align*}
			and 
			\[\{a_h + a_3, a_h + a_4\} \cap (B_{1, 2} \cup B_{2, 2}) = \emptyset,\]
			it follows that 
			\[a_h + a_3, a_h + a_4 \in C_2^{(0)}.\] 
			Hence $|C_2^{(0)}| \geq 2$. Therefore, we have
			\[\alpha_{0} = |C_1^{(0)}| + |C_2^{(0)}| \geq 3 + 2 = 5.\]
			This proves \eqref{rssn-lemeq:7} for $t = 0$.
		\end{enumerate}
		
		\noindent {\textbf{Case 2.}} ($h \geq 6$ and $t \in [1, h - 5]$). Let 
		\[y = a_{h + 1} + \cdots + a_{h + 2 - t}.\]
		Then
		\begin{enumerate}
			\item If $a_{h + 1 - t} - a_{h - t} = a_2$, then we have $a_{h + 1 - t} + a_2 = a_{h - t} + a_3$. Since
			\[y + a_{h - t} + a_3 \in \Sigma (A_2),\]
			\begin{align*}
				{\max}(B_{t + 1, 2}) &= y + a_{h - t} < y + a_{h + 1 - t} = \max (B_{t + 1, 2})\\
				&< y + a_{h - t} + a_3 < y + a_{h + 1 - t} + a_3 = \min (B_{t + 2, 2}) \\
				&< y + a_{h + 1 - t} + a_{h - 1 - t} = {\max}_{-} (B_{t + 2, 2}),
			\end{align*}
			and 
			\[y + a_{h - t} + a_3 \notin B_{t + 1, 2} \cup B_{t + 2, 2},\]
			it follows that 
			\[y + a_{h - t} + a_3 \in C_2^{(t)}.\] 
			Hence $|C_2^{(t)}| \geq 1$. Therefore, we have
			\[\alpha_t \geq |C_2^{(t)}| \geq 1.\]
			This proves \eqref{rssn-lemeq:1} for each $t \in [1, h - 5]$.
			
			\item If $a_{h + 1 - t} - a_{h - t} = a_3$, then we have $a_{h + 1 - t} + a_2 = a_{h - t} + a_4$. Since
			\[a_{h - t} + a_2, y + a_{h - t} + a_2 \in \Sigma (A_1),\]
			\[a_{h - t} < a_{h - t} + a_2 < a_{h + 1- t},\]
			\begin{align*}
				{\max}_{-} (B_{t + 1, 1}) &= y + a_{h - t} < y + a_{h - t} + a_2 < y + a_{h + 1 - t} \\
				&= \max (B_{t + 1, 1}) < y + a_{h + 1 - t} + a_2 = \min (B_{t + 2, 1}) \\
				&< y + a_{h + 1 - t} + a_{h - 1 - t} = {\max}_{-} (B_{t + 2, 1}),
			\end{align*}
			and
			\[y + a_{h - t} + a_2 \notin B_{t + 1, 1} \cup B_{t + 2, 1},\]
			it follows that 
			\[a_{h - t} + a_2, y + a_{h - t} + a_2 \in C_1^{(t)}.\] 
			Hence $|C_1^{(t)}| \geq 2$. Since
			\[y + a_{h - t} + a_4 \in \Sigma (A_2),\]
			\begin{align*}
				{\max}(B_{t + 1, 2}) &= y + a_{h - t} < y + a_{h + 1 - t} = \max (B_{t + 1, 2})\\
				&< y + a_{h - t} + a_4 < y + a_{h + 1 - t} + a_3 = \min (B_{t + 2, 2}) \\
				&< y + a_{h + 1 - t} + a_{h - 1 - t} = {\max}_{-} (B_{t + 2, 2}),
			\end{align*}
			and 
			\[y + a_{h - t} + a_4 \notin B_{t + 1, 2} \cup B_{t + 2, 2},\]
			it follows that 
			\[y + a_{h - t} + a_3 \in C_2^{(t)}.\] 
			Hence $|C_2^{(t)}| \geq 1$. Therefore, we have
			\[\alpha_t = |C_1^{(t)}| + |C_2^{(t)}| \geq 2 + 1 = 3.\]
			This proves \eqref{rssn-lemeq:2} for each $t \in [1, h - 5]$.

			\item If $a_{h + 1 - t} - a_{h - t} = a_3$ and $a_{h - t} - a_{h - 1 - t} = a_2$, then as proved earlier 
			\[a_{h - t} + a_2, y + a_{h - t} + a_2 \in C_1^{(t)}.\] 
			Hence $|C_1^{(t)}| \geq 2$. Since $a_{h - t} < a_{h - 1 - t} + a_3 < a_{h + 1 - t}$, it follows that
			\[a_{h - 1 - t} + a_3 \in C_2^{(t)}.\]
			It was already shown that $y + a_{h - t} + a_4 \in C_2^{(t)}$. Since 
			\[a_{h - 1 - t} + a_3 \neq y + a_{h - t} + a_4,\]
			we have $|C_2^{(t)}| \geq 2$.	Therefore, we have
			\[\alpha_t \geq |C_1^{(t)}| + |C_2^{(t)}| \geq 2 + 2 = 4.\]
			This proves \eqref{rssn-lemeq:5} for $t \in [1, h - 5]$.
			
			\item If $a_{h + 1 - t} - a_{h - t} + a_4$, then we have 
			\[a_{h - t} < a_{h - t} + a_2 < a_{h - t} + a_3 < a_{h + 1 - t}.\] 
			Since 
			\[a_{h - t} + a_2, a_{h - t} + a_3, y + a_{h - t} + a_2, y + a_{h - t} + a_3 \in \Sigma(A_1),\]
			\[a_{h - t} < a_{h - t} + a_2 < a_{h - t} + a_3 < a_{h + 1 - t},\] 
			\begin{align*}
				{\max}_{-} (B_{t + 1, 1}) &= y + a_{h - t} < y + a_{h - t} + a_2 < y + a_{h - t} + a_3\\
				&< y + a_{h + 1 - t} = \max (B_{t + 1, 1}) < y + a_{h + 1 - t} + a_2  \\
				&= \min (B_{t + 2, 1}) < y + a_{h + 1 - t} + a_{h - 1 - t} = {\max}_{-} (B_{t + 2, 1}),
			\end{align*}
			and
			\[y + a_{h - t} + a_2, y + a_{h - t} + a_3 \notin B_{t + 1, 1} \cup B_{t + 2, 1},\]
			it follows that 
			\[a_{h - t} + a_2, a_{h - t} + a_3, y + a_{h - t} + a_2, y + a_{h - t} + a_3 \in C_1^{(t)}.\] 
			Hence $|C_1^{(t)}| \geq 4$. Since
			\[a_{h - t} + a_3, y + a_{h - t} + a_3 \in \Sigma (A_2),\]
			\[a_{h - t} < a_{h - t} + a_3 < a_{h + 1 - t},\]
			\begin{align*}
				{\max}(B_{t + 1, 2}) &= y + a_{h - t} < y + a_{h - t} = a_3 < y + a_{h + 1 - t} \\
				&= \max (B_{t + 1, 2}) < y + a_{h + 1 - t} + a_3 = \min (B_{t + 2, 2}) \\
				&< y + a_{h + 1 - t} + a_{h - 1 - t} = {\max}_{-} (B_{t + 2, 2}),
			\end{align*}
			and 
			\[y + a_{h - t} + a_3 \notin B_{t + 1, 2} \cup B_{t + 2, 2},\]
			it follows that 
			\[a_{h - t} + a_3, y + a_{h - t} + a_3 \in C_2^{(t)}.\] 
			Hence $|C_2^{(t)}| \geq 2$. Therefore, we have
			\[\alpha_t = |C_1^{(t)}| + |C_2^{(t)}| \geq 4 + 2 = 6.\]
			This proves \eqref{rssn-lemeq:6} for each $t \in [1, h - 5]$.
		   
			\item Now we assume that $a_{h + 1 - t} - a_{h - t} \notin \{a_2, a_3, a_4\}$. It is easy to see that sets $\{y + a_{h - t} + a_2, x + a_{h - t} + a_3, y + a_{h - t} + a_4\}$ and $\{y + a_{h + 1 - t}, y + a_{h + 1 - t} + a_2, y + a_{h + 1 - t} + a_3\}$ are disjoint. Since 
			\[y + a_{h - t} + a_2, y + a_{h - t} + a_3, y + a_{h - t} + a_4 \in \Sigma(A_1),\]
			\begin{align*}
				{\max}_{-} (B_{t + 1, 1}) &= y + a_{h - t} < y + a_{h - t} + a_2 < y + a_{h - t} + a_3\\
				&< y + a_{h - t} + a_4 < y + a_{h + 1 - t} + a_{h - 1 - t} = {\max}_{-} (B_{t + 2, 1}),\\
			\end{align*}
			and
			\[y + a_{h - t} + a_2, y + a_{h - t} + a_3, y + a_{h - t} + a_4 \notin B_{t + 1, 1} \cup B_{t + 2, 1},\]
			it follows that 
			\[y + a_{h - t} + a_2, y + a_{h - t} + a_3, y + a_{h - t} + a_4 \in C_1^{(t)}.\] 
			Hence $|C_1^{(t)}| \geq 3$. Since
			\[y + a_{h - t} + a_3, y + a_{h - t} + a_4 \in \Sigma (A_2),\]
			\begin{align*}
				{\max}(B_{t + 1, 2}) &= y + a_{h - t} < y + a_{h - t} + a_3 < y + a_{h - t} + a_4 \\
				&< y + a_{h + 1 - t} + a_{h - 1 - t} = {\max}_{-} (B_{t + 2, 2}),
			\end{align*}
			and 
			\[y + a_{h - t} + a_3, y + a_{h - t} + a_4 \notin B_{t + 1, 2} \cup B_{t + 2, 2},\]
			it follows that 
			\[y + a_{h - t} + a_3, y + a_{h - t} + a_4 \in C_2^{(t)}.\] 
			Hence $|C_2^{(t)}| \geq 2$. Therefore, we have
			\[\alpha_t = |C_1^{(t)}| + |C_2^{(t)}| \geq 3 + 2 = 5.\]
			This proves \eqref{rssn-lemeq:7} for each $t \in [1, h - 5]$.
		\end{enumerate} 
		
		\noindent {\textbf{Case 3}} ($t = h - 4$). Let 
		\[x = a_{h + 1} + \cdots + a_6.\]
		We have
		\begin{enumerate}
			\item If $a_5 = a_4 + a_2$, then we have $a_5 + a_2 = a_4 + a_3$. Since
			\[x + a_4 + a_3 \in \Sigma (A_2),\]
			\begin{align*}
				{\max}_{-} (B_{h - 3, 2}) &= x + a_4 < \max (B_{h - 3, 2}) = x + a_5 < x + a_4 + a_3\\
				&< x + a_5 + a_3 = \min (B_{h - 2, 2}) < x + a_5 + a_4 = \max (B_{h - 2, 2}),
			\end{align*}
			and
			\[x + a_4 + a_3 \notin B_{h- 3, 2} \cup B_{h - 2, 2},\]
			it follows that 
			\[x + a_4 + a_3 \in C_2^{(h - 4)}.\] 
			Hence $|C_2^{(h - 4)}| \geq 1$. Therefore, we have
			\[\alpha_{h - 4} \geq |C_1^{(h - 4)}| + |C_2^{(h - 4)}| \geq 1.\]
			This proves \eqref{rssn-lemeq:1} for $t = h - 4$.
			
			\item If $a_5 - a_4 = a_3$, then we have $a_4 < a_4 + a_2 < a_5$. Since
			\[a_4 + a_2, x + a_4 + a_2 \in \Sigma (A_1),\]
			\[a_4 < a_4 + a_2 < a_5,\]
			\begin{align*}
				{\max}_{-} (B_{h - 3, 1}) &= x + a_4 < x + a_4 + a_2 < x + a_5 = \max (B_{h - 3, 1}) \\
				&< x + a_5 + a_2 = \min (B_{h - 2, 2}) < x + a_5 + a_4 = \max (B_{h - 2, 1}),
			\end{align*}
			and
			\[x + a_4 + a_2 \notin B_{h - 3, 1} \cup B_{h - 2, 1}),\]
			it follows that 
			\[a_4 + a_2, x + a_4 + a_2 \in C_1^{(h - 4)}.\] 
			Hence $|C_1^{(h - 4)}| \geq 2$. Therefore, we have
			\[\alpha_{h - 4} \geq |C_1^{(h - 4)}| + |C_2^{(h - 4)}| \geq 2.\]
			This proves \eqref{rssn-lemeq:2} for $t = h - 4$.
			
			\item If $a_5 = a_4 + a_4$, then 
			\[a_4 < a_4 + a_2 < a_4 + a_3 < a_5.\] 
			Since
			\[a_4 + a_2, a_4 + a_3, x + a_4 + a_2, x + a_4 + a_3 \in \Sigma (A_1),\] 
			\[a_4 < a_4 + a_2 < a_4 + a_3 < a_5,\]
			\begin{align*}
				{\max}_{-} (B_{h - 3, 1}) &= x + a_4 < x + a_4 + a_2 < x + a_4 + a_3 < x + a_5 \\
				&= \max (B_{h - 3, 1}) < x + a_5 + a_4 = \max (B_{h - 2, 1}),
			\end{align*}
			and
			\[x + a_4 + a_2, x + a_4 + a_3 \notin B_{h - 3, 1} \cup B_{h - 2, 1},\]
			it follows that 
			\[a_4 + a_2, a_4 + a_3, x + a_4 + a_2, x + a_4 + a_3, \in C_1^{(h - 4)}.\] 
			Hence $|C_1^{(h - 4)}| \geq 4$. Since
			\[a_4 + a_3, x + a_4 + a_3 \in \Sigma (A_2),\]
			\[a_4 < a_4 + a_3 < a_5,\]
			\begin{align*}
				{\max}_{-} (B_{h - 3, 2}) &= x + a_4 < x + a_4 + a_3 < x + a_5\\
				&= \max (B_{h - 3, 2}) < x + a_5 + a_4 = \max (B_{h - 2, 2}),
			\end{align*}
			and
			\[x + a_4 + a_3 \notin B_{h- 3, 2} \cup B_{h - 2, 2},\]
			it follows that 
			\[a_4 + a_3, x + a_4 + a_3 \in C_2^{(h - 4)}.\] 
			Hence $|C_2^{(h - 4)}| \geq 2$. Therefore, we have
			\[\alpha_{h - 4} \geq |C_1^{(h - 4)}| + |C_2^{(h - 4)}| \geq 4 + 2 = 6.\]
			This proves \eqref{rssn-lemeq:6} for $t = h - 4$.

			\item If $a_5 - a_4 \notin \{a_2, a_3\}$, then it is easy to see that 
			\[x + a_5 \notin \{x + a_4 + a_2, x + a_4 + a_3\}.\]
			Since
			\[x + a_4 + a_2, x + a_4 + a_3 \in \Sigma (A_1),\] 
			\begin{align*}
				{\max}_{-} (B_{h - 3, 1}) &= x + a_4 < x + a_4 + a_2 < x + a_4 + a_3 < x + a_5 + a_3\\
				&= \min (B_{h - 2, 1}) < x + a_5 + a_4 = \max (B_{h - 2, 1}),
			\end{align*}
			and
			\[x + a_4 + a_2, x + a_4 + a_3 \notin B_{h - 3, 1} \cup B_{h - 2, 1},\]
			it follows that 
			\[x + a_4 + a_2, x + a_4 + a_3, \in C_1^{(h - 4)}.\] 
			Hence $|C_1^{(h - 4)}| \geq 2$. Since
			\[x + a_4 + a_3 \in \Sigma (A_2),\]
			\begin{align*}
				{\max}_{-} (B_{h - 3, 2}) &= x + a_4 < x + a_4 + a_3 < x + a_5 + a_3\\
				&= \min (B_{h - 2, 2}) < x + a_5 + a_4 = \max (B_{h - 2, 2}),
			\end{align*}
			and
			\[x + a_4 + a_3 \notin B_{h- 3, 2} \cup B_{h - 2, 2},\]
			it follows that 
			\[x + a_4 + a_3 \in C_2^{(h - 4)}.\] 
			Hence $|C_2^{(h - 4)}| \geq 1$. Therefore, we have
			\[\alpha_{h - 4} \geq |C_1^{(h - 4)}| + |C_2^{(h - 4)}| \geq 2 + 1 = 3.\]
			This proves \eqref{rssn-lemeq:7} for $t = h - 4$.
		
			\item If $a_5 - a_4 \notin \{a_2, a_3\}$ and $a_6 = a_5 + a_2$, then it is easy to see that 
			\[a_4 < a_5 \neq a_4 + a_2 < a_6.\]
			\begin{enumerate}
				\item [(i)] If $a_4 <  a_4 + a_2 < a_5$, then we have
				\[a_4 + a_2 \in C_1^{(h - 4)}.\]
				Also, we proved earlier that
				\[x + a_4 + a_2, x + a_4 + a_3, \in C_1^{(h - 4)},\]
				and
				\[x + a_4 + a_3 \in C_2^{(h - 4)}.\]
				Since $a_4 + a_2 \notin \{x + a_4 + a_2, x + a_4 + a_3\}$, we have
				\[|C_1^{(h - 4)}| \geq 3 ~\text{and}~ |C_2^{(h - 4)}| \geq 1.\]
				Hence
				\[\alpha_{h - 4} + \alpha_{h - 5} \geq |C_1^{(h - 4)}| + |C_2^{(h - 4)}| + 1 \geq 3 + 1 + 1 = 5.\]
				\item [(ii)] If $a_5 < a_4 + a_2 < a_6$, then we have
				\[a_4 + a_2 \in C_1^{(h - 5)}.\]
				Clearly, if $h = 5$, then we have
				\[a_4 + a_3 \in \Sigma (A_2),\]
				and
				\[a_4 < a_5 < a_6 \neq a_4 + a_3 < a_6 + a_3.\]
				Hence $a_4 + a_2 \in C_2^{(h - 5)}$. If $h \geq 6$, then we have
				\[a_{h + 1} + \cdots + a_7 + a_4 + a_3 \in \Sigma (A_2),\]
				\begin{align*}
					{\max}_{-} (B_{h - 4, 2}) &= a_{h + 1} + \cdots + a_7 + a_5 \\
					&< a_{h + 1} + \cdots + a_7 + a_6  \neq a_{h + 1} + \cdots + a_7 + a_4 + a_3\\
					&< x + a_3 = \min (B_{h - 3, 2}) < {\max}_{-} (B_{h - 2, 2}),
				\end{align*}
				\[a_{h + 1} + \cdots + a_7 + a_4 + a_3 \notin B_{h - 4, 2} \cup B_{h - 3, 2}\]
				Hence $$a_{h + 1} + \cdots + a_7 + a_4 + a_3 \in C_2^{(h - 5)},$$ and so
				\[|C_1^{(h - 5)}| \geq 1 ~\text{and}~ |C_2^{(h - 5)}| \geq 1.\]
				Therefore,
				\[\alpha_{h - 5} \geq |C_1^{(h - 5)}| + |C_2^{(h - 5)}| \geq 1 + 1 = 2.\]
				Since $a_5 - a_4 \notin \{a_2, a_4\}$, we proved earlier $\alpha_{h - 4} \geq 3$, it follows that 
				\[\alpha_{h - 4} + \alpha_{h - 5} \geq 3 + 2 = 5.\]
			\end{enumerate}
			 Therefore, 
			\[\alpha_{h - 4} + \alpha_{h - 5} \geq 5.\]
		\end{enumerate}
		This proves \eqref{rssn-lemeq:8} for $t = h - 4$. 
		This completes the proof.
	\end{proof}
The following Lemma follows from Theorem $4$ in \cite{mmp2024}, but our proof is slightly different.
\begin{lemma}\label{rssn-lem18}
	Let $h \geq 5$ be an integer. Let $A = [0, h]$ be a set of nonnegative integers. Then, we have
	\[h_{\pm}^\wedge A = \biggl [- \frac{h(h + 1)}{2}, \frac{h(h + 1)}{2} \biggl].\]
	Therefore,
	\[|h_{\pm}^\wedge A| = h^2 + h + 1.\]
\end{lemma}

\begin{proof}
	By applying Lemma \ref{rssn-lem16} (See $(1)$), we get
	\[\sum_{j = 0}^{h - 4} \alpha_j \geq h - 3.\]
	Hence
	\[|h_{\pm}^\wedge A| \geq |h_{\pm}^\wedge A_1| + |h_{\pm}^\wedge A_2| = \sum_{j = 0}^{h - 4} \alpha_j + h^2 + 2 = h^2 + h - 1.\] 
	Let 
	\[v = 0 - 1 - 2 - 4 - \cdots - h.\]
	Then 
	\[v \not \equiv x \pmod2,\]
	where $x \in h_{\pm}^\wedge A_1$. Hence
	\[v \notin h_{\pm}^\wedge A_1.\]
	Since $\min (h_{\pm}^\wedge A_2) = 0 - 2 - \cdots - h$ and ${\min}_{+} (h_{\pm}^\wedge A_2) = 0 + 2 - 3 \cdots - h$, it follows that
	\[\min (h_{\pm}^\wedge A_2) < v < {\min}_{+} (h_{\pm}^\wedge A_2).\] 
	Hence
	\[v \notin h_{\pm}^\wedge A_2.\]
	Therefore,
	\[v \notin h_{\pm}^\wedge A_1 \cup h_{\pm}^\wedge A_2.\]
	Since $h_{\pm}^\wedge A_1 \cup h_{\pm}^\wedge A_2$ is a symmetric set and $v \notin h_{\pm}^\wedge A_1 \cup h_{\pm}^\wedge A_2$, it follows that
	\[- v \notin h_{\pm}^\wedge A_1 \cup h_{\pm}^\wedge A_2.\]
	Since $- v, v \in h_{\pm}^\wedge A\setminus (h_{\pm}^\wedge A_1 \cup h_{\pm}^\wedge A_2)$, it follows that 
\begin{equation}\label{rssn-lem18eq1}
|h_{\pm}^\wedge A| \geq |h_{\pm}^\wedge A_1 \cup h_{\pm}^\wedge A_2| + 2 = |h_{\pm}^\wedge A_1| + |h_{\pm}^\wedge A_2| + 2 \geq h^2 + h + 1.
\end{equation}
Since 
\begin{equation}\label{rssn-lem18eq2}
h_{\pm}^\wedge A \subseteq \biggl [- \frac{h(h + 1)}{2}, \frac{h(h + 1)}{2} \biggl],
\end{equation}
we have
\begin{equation}\label{rssn-lem18eq3}
|h_{\pm}^\wedge A| \leq  h^2 + h +1.
\end{equation}
Hence it follows from $\eqref{rssn-lem18eq1}$, $\eqref{rssn-lem18eq2}$ and $\eqref{rssn-lem18eq3}$ that
	\[|h_{\pm}^\wedge A| = h^2 + h + 1,\]
	and
	\[h_{\pm}^\wedge A = \biggl [- \frac{h(h + 1)}{2}, \frac{h(h + 1)}{2} \biggl].\]
\end{proof}

\begin{lemma}\label{rssn-lem21}
	Let $h \geq 5$ be an integer. Let $A = \{0, 1, 2, 3, 5, 6, \ldots, h + 1\}$ be a set of nonnegative integers. Then, we have
	\[|h_{\pm}^\wedge A| \geq h^2 + h + 2.\]
\end{lemma}

\begin{proof}
	By applying Lemma \ref{rssn-lem16} (See $(1)$ and $(2)$), we get
	\[\alpha_{h - 4} \geq 2.\]
	Therefore,
	\[\sum_{j = 0}^{h - 5} \alpha_j + \alpha_{h - 4} \geq (h - 4) + 2 = h - 2.\]
	Since $|h_{\pm}^\wedge A| \geq |h_{\pm}^\wedge A_1| + |h_{\pm}^\wedge A_2| = \sum_{j = 0}^{h - 4} \alpha_j + h^2 + 2$, it follows that 
	\[|h_{\pm}^\wedge A| \geq |h_{\pm}^\wedge A_1| + |h_{\pm}^\wedge A_2| \geq h ^2 + h.\]
	Let
	\[v = 0 - 1 - 2 - 5 - 6 - \cdots - (h + 1).\]
	Then, as similar Lemma \ref{rssn-lem18}, we show that
	\[- v, v \notin h_{\pm}^\wedge A_1 \cup h_{\pm}^\wedge A_2.\]
	Since $- v, v \in h_{\pm}^\wedge A \setminus (h_{\pm}^\wedge A_1 \cup h_{\pm}^\wedge A_2)$, it follows that
	\[|h_{\pm}^\wedge A| \geq |h_{\pm}^\wedge A_1 \cup h_{\pm}^\wedge A_2| + 2 = |h_{\pm}^\wedge A_1| + |h_{\pm}^\wedge A_2| + 2 \geq h^2 + h + 2.\]
	Hence
	\[|h_{\pm}^\wedge A| = |h_{\pm}^\wedge A| \geq h^2 + h + 2.\]	
	This completes the proof.
\end{proof}

\begin{lemma}\label{rssn-lem22}
	Let $h \geq 5$ be an integer. Let $A = \{0, 1, 2, 3, \ldots, h - 1, h + 1\}$ be a set of nonnegative integers. Then, we have
	\[|h_{\pm}^\wedge A| \geq h^2 + h + 3.\]
\end{lemma}

\begin{proof}
	By applying Lemma \ref{rssn-lem16} (See $(1)$ and $(3)$), we get
	\[\alpha_0 \geq 3.\]
	Therefore,
	\[\sum_{j = 0}^{h - 4} = \alpha_0 + \sum_{j = 0}^{h - 5} \alpha_j \geq 3 + (h - 4) + 2 = h - 1.\]
	Since $|h_{\pm}^\wedge A| \geq |h_{\pm}^\wedge A_1| + |h_{\pm}^\wedge A_2| = \sum_{j = 0}^{h - 4} \alpha_j + h^2 + 2$, it follows that 
	\[|h_{\pm}^\wedge A| \geq |h_{\pm}^\wedge A_1| + |h_{\pm}^\wedge A_2| \geq h ^2 + h + 1.\]
	Let
	\[v = 0 - 1 - 2 - 4 - \cdots - (h - 1) - (h + 1).\]
	Then, as similar Lemma \ref{rssn-lem18}, we show that
	\[- v, v \notin h_{\pm}^\wedge A_1 \cup h_{\pm}^\wedge A_2.\]
	Since $- v, v \in h_{\pm}^\wedge A \setminus (h_{\pm}^\wedge A_1 \cup h_{\pm}^\wedge A_2)$, it follows that
	\[|h_{\pm}^\wedge A| \geq |h_{\pm}^\wedge A_1 \cup h_{\pm}^\wedge A_2| + 2 = |h_{\pm}^\wedge A_1| + |h_{\pm}^\wedge A_2| + 2 \geq h^2 + h + 3.\]
	Hence
	\[|h_{\pm}^\wedge A| = |h_{\pm}^\wedge A| \geq h^2 + h + 3.\]	
	This completes the proof.
\end{proof}
	
	\begin{lemma}\label{rssn-lem17}
		Let $h \geq 5$ be an integer. Let $A = \{a_1, a_2, \ldots, a_{h + 1}\}$ be a set of nonnegative integers such that $a_1 = 0 < a_2 < \cdots < a_{h + 1}$. Furthermore, assume that
		\[a_2 \not \equiv 0 \pmod {2}, ~a_3 = 2a_2, ~\text{and}~ a_4 = 3a_2.\]
		Then, we have
		\[|h_{\pm}^\wedge A| \geq h^2 + h + 2,\]
		when $A \neq a_2 \ast [0, h]$. Otherwise, 
		\[|h_{\pm}^\wedge A| = h^2 + h + 1.\]
	\end{lemma}
	
	\begin{proof}
		Since $|h_{\pm}^\wedge A| \geq h^2 + 2 + \sum_{n = 0}^{h - 4} \alpha_n$, it follows that if $A \neq a_2 \ast [0, h]$, then it is sufficient to show that
		\[\sum_{n = 0}^{h - 4} \alpha_n \geq h.\]
		Now we consider the following case.
		
		\noindent {\textbf{Case 1}} ($a_{i + 1} - a_i \notin \{a_2, a_3, a_4\}$ for some $i \in [4, h]$). Now we consider the following case.
		
		\noindent {\textbf{Case 1.1}} ($i = 4$). In this case, we have 
		\[a_5 - a_4 \notin \{a_2, a_3, a_4\}.\]
		By applying Lemma \ref{rssn-lem16} (See $(1)$ and $(5)$), we get
		\[\alpha_{h - 4} \geq 3.\]
		\begin{enumerate}
			\item If $a_6 - a_5 = a_2$, then by applying Lemma \ref{rssn-lem16} (See $(1)$ and $(6)$), we get
			\[\alpha_{h - 4} + \alpha_{h - 5} \geq 5.\]
			Hence
			\[\sum_{n = 0}^{h - 4} \alpha_n = \sum_{\substack{n = 0 \\ n \neq h - 4, h - 5}}^{h - 4} \alpha_n + \alpha_{h - 4} + \alpha_{h - 5} \geq (h - 5) + 5 = h.\]
			\item If $a_6 - a_5 = a_3$, then by applying Lemma \ref{rssn-lem16} (See $(1)$, $(2)$ and $(5)$), we get
			\[\alpha_{h - 4} \geq 3 ~\text{and}~ \alpha_{h - 5} \geq 2.\]
			Hence
			\[\sum_{n = 0}^{h - 4} \alpha_n = \sum_{\substack{n = 0 \\ n \neq h - 4, h - 5}}^{h - 4} \alpha_n + \alpha_{h - 4} + \alpha_{h - 5} \geq (h - 5) + 3 + 2 = h.\]
			\item If $a_6 - a_5 = a_4$, then by applying Lemma \ref{rssn-lem16} (See $(1)$, $(4)$ and $(5)$), we get
			\[\alpha_{h - 4} \geq 3 ~\text{and}~ \alpha_{h - 5} \geq 3.\]
			Hence
			\[\sum_{n = 0}^{h - 4} \alpha_n = \sum_{\substack{n = 0 \\ n \neq h - 4, h - 5}}^{h - 4} \alpha_n + \alpha_{h - 4} + \alpha_{h - 5} \geq (h - 5) + 3 + 3 = h + 1.\]
			\item If $a_6 - a_5 \notin \{a_2, a_3, a_4\}$, then by applying Lemma \ref{rssn-lem16} (See $(1)$ and $(5)$), we get
			\[\alpha_{h - 4} \geq 3 ~\text{and}~ \alpha_{h - 5} \geq 5.\]
			Hence
			\[\sum_{n = 0}^{h - 4} \alpha_n = \sum_{\substack{n = 0 \\ n \neq h - 4, h - 5}}^{h - 4} \alpha_n + \alpha_{h - 4} + \alpha_{h - 5} \geq (h - 5) + 3 + 5 = h + 3.\]
		\end{enumerate}
		
		\noindent {\textbf{Case 1.2}} ($i \in [5, h]$). By applying Lemma \ref{rssn-lem16} (See $(1)$ and $(5)$), we get
		\[\alpha_{h - 4} \geq 5.\]
		Hence
		\[\sum_{j = 0}^{h - 4} \alpha_j = \sum_{j = 0}^{h - 5} \alpha_j + \alpha_{h - 4} \geq (h - 4) + 5 = h + 1.\]
		
		\noindent {\textbf{Case 2}} ($a_{i + 1} - a_i \in \{a_2, a_3, a_4\}$ for each $i \in [4, h]$). Let 
		\begin{align*}
			r &= |\{i \in [4, h] : a_{i + 1} - a_i = a_2\}|,\\
			s &= |\{i \in [4, h] : a_{i + 1} - a_i = a_3\}|,\\
			t &= |\{i \in [4, h] : a_{i + 1} - a_i = a_3\}|.
		\end{align*}
		Then 
		\[r + s + t = h - 3 \geq 2.\]
		Now we consider the following case.
		
		\noindent {\textbf{Case 2.1}} ($t \geq 1$). Let $a_{i + 1} - a_i = a_4$ for some $i \in [4, h]$. Then by applying Lemma \ref{rssn-lem16} (See $(1)$ and $(4)$), we get
		\begin{equation}
			\alpha_{h - i} \geq 
			\begin{cases}
				3, &\text{if $i = h$};\\
				6, &\text{otherwise}.  
			\end{cases}
		\end{equation} 
		If $4 \leq i < h$, then by applying Lemma \ref{rssn-lem16} (See $(1)$ and $(4)$), we get
		\[\alpha_{h - 4} \geq 6.\]
		Hence
		\[\sum_{n = 0}^{h - 4} \alpha_n = \sum_{\substack{n = 0 \\ n \neq h - i}}^{h - 4} \alpha_n + \alpha_{h - i} \geq (h - 4) + 6 = h + 2.\]
		 If $i = h$, then 
		 \[a_{h + 1} - a_h = a_4.\]
		 Now
		 \begin{enumerate}
		 	 \item If $a_h - a_{h - 1} = a_2$, then by applying Lemma \ref{rssn-lem16} (See $(1)$ and $(7)$), we get
		 	 \[\alpha_0 \geq 4.\]
		 	 Hence
		 	 \[\sum_{n = 0}^{h - 4} \alpha_n = \alpha_0 + \sum_{n = 1}^{h - 4} \alpha_n \geq 4 + (h - 4) = h.\]
		 	 \item If $a_h - a_{h - 1} = a_3$, then by applying Lemma \ref{rssn-lem16} (See $(1)$, $(2)$ and $(4)$), we get
		 	 \[\alpha_0 \geq 3 ~\text{and}~ \alpha_1 \geq 2.\]
		 	 Hence
		 	 \[\sum_{n = 0}^{h - 4} \alpha_n = \alpha_0 + \alpha_1 + \sum_{\substack{n = 0 \\ n \neq 0, 1}}^{h - 4} \alpha_n \geq 3 + 2 + (h - 5) = h.\]
		 	 \item If $a_h - a_{h - 1} = a_4$, then by applying Lemma \ref{rssn-lem16} (See $(1)$ and $(4)$), we get
		 	 \[\alpha_0 \geq 3 ~\text{and}~ \alpha_1 \geq 6.\]
		 	 Hence
		 	 \[\sum_{n = 0}^{h - 4} \alpha_n = \alpha_0 + \alpha_1 + \sum_{\substack{n = 0 \\ n \neq 0, 1}}^{h - 4} \alpha_n \geq 3 + 6 + (h - 5) = h + 4.\]
		 	 \item ($a_h - a_{h - 1} \notin \{a_2, a_3, a_4\}$). This case not possible.	
		 \end{enumerate}
		 
		 \noindent {\textbf{Case 2.2}} ($t = 0$).
		 Since $t = 0$, it follows that
		 \[a_{i + 1} - a_i \in \{a_2, a_3\}\]
		 for each $i \in [4, h]$. Now we consider the following subcase.
		 
		 \noindent {\textbf{Subcase 2.2.1}} ($t = 0$ ~\text{and}~ $s \geq 3$).
		 Since $s \geq 3$, it follows that there exist $i, j, k \in [4, h]$ such that 
		 $i < j <k$ and 
		 \[a_{i + 1} - a_i = a_{j + 1} - a_j = a_{k + 1} - a_k = a_3.\]
		 By by applying Lemma \ref{rssn-lem16} (See $(1)$ and $(2)$), we get
		 \[\alpha_{h - i} \geq 2, \alpha_{h - j} \geq 3 ~\text{and}~ \alpha_{h - k} \geq 2.\]
		 Hence
		 \[\sum_{n = 0}^{h - 4} \alpha_n = \sum_{\substack{n = 0 \\ n \neq h - i, h - j, h - k}}^{h - 4} \alpha_n + \alpha_{h - i} + \alpha_{h - j} + \alpha_{h - k} \geq (h - 6) + 2 + 3 + 2 = h + 1.\]
		 
		 \noindent {\textbf{Subcase 2.2.2}} ($t = 0$ ~\text{and}~ $s = 2$).
		 Since $s = 2$, it follows that there exist exactly two integer $i, j \in [4, h]$ such that $i < j$ and 
		 \[a_{i + 1} - a_i = a_{j + 1} - a_j = a_3.\]
		 Also,
		 \[a_{k + 1} - a_k = a_2,\]
		 where $k \in [1, h] \setminus \{i, j\}$. Now by applying Lemma \ref{rssn-lem16} (See $(1)$ and $(2)$), we get
		 \[\alpha_{h - i} \geq 2 ~\text{and}~ \alpha_{h - j} \geq 2.\]
		 \begin{enumerate}
		 	\item If $h = 5$, then 
		 	\[a_6 - a_5 = a_3 = a_5 - a_4.\]
		 	Hence
		 	\[A = a_2 \ast \{0, 1, 2, 3, 5, 7\}.\]
		 	Since $|h_{\pm}^\wedge A| = |h_{\pm}^\wedge \{0, 1, 2, 3, 5, 7\}|$ and $|h_{\pm}^\wedge \{0, 1, 2, 3, 5, 7\}| \geq 32$, it follows that
		 	\[|h_{\pm}^\wedge A| \geq 32.\]
		 	\item If $h \geq 6$, $i = 4$ and $j = h$, then 
		 	\[a_h - a_{h -1} = a_2.\]
		 	By applying Lemma \ref{rssn-lem16} (See $(1)$, $(2)$ and $(3)$), we get
		 	\[\alpha_0 \geq 3 ~\text{and}~ \alpha_{h - 4} \geq 2.\]
		 	Hence
		 	\[\sum_{n = 0}^{h - 4} \alpha_n = \alpha_0 + \sum_{\substack{n = 0 \\ n \neq 0, h - 4}}^{h - 4} \alpha_n + \alpha_{h - 4} \geq 3 + (h - 5) + 2 = h.\]
		 	\item If $4 \leq i < j < h$, then by applying Lemma \ref{rssn-lem16} (See $(1)$ and $(2)$), we get
		 	\[\alpha_{h - i} \geq 2 ~\text{and}~ \alpha_{h - j} \geq 3.\]
		 	Hence
		 	\[\sum_{n = 0}^{h - 4} \alpha_n = \sum_{\substack{n = 0 \\ n \neq h - i, h - j}}^{h - 4} \alpha_n + \alpha_{h - i} + \alpha_{h - j} \geq (h - 5) + 2 + 3 = h.\]
		 \end{enumerate}
		 
		 \noindent {\textbf{Subcase 2.2.3}} ($t = 0$ ~\text{and}~ $s = 1$).
		 Since $s = 1$, it follows that there exist  $i_0 \in [4, h]$ such that  
		 \[a_{i_0 + 1} - a_{i_0}= a_3 \]
		 and
		 \[a_{i + 1} - a_i = a_2\]
		 for each $i \in [1, h] \setminus \{i_0\}$. 
		 \begin{enumerate}
		 	\item If $i_0 = 4$, then
		 	\[A = a_2 \ast \{0, 1, 2, 3, 5, 6, \ldots, h + 1\}.\]
		 	Since $|h_{\pm}^\wedge A| = |h_{\pm}^\wedge \{0, 1, 2, 3, 5, 6, \ldots, h + 1\}|$, it follows from Lemma \ref{rssn-lem21} that 
		 	\[|h_{\pm}^\wedge A| \geq h^2 + h + 2.\]
		 	\item f $i_0 = h$, then
		 	\[A = a_2 \ast \{0, 1, 2, 3, \ldots, h - 1, h + 1\}.\]
		 	Since $|h_{\pm}^\wedge A| = |h_{\pm}^\wedge \{0, 1, 2, 3, \ldots, h - 1, h + 1\}|$, it follows from Lemma \ref{rssn-lem22} that 
		 	\[|h_{\pm}^\wedge A| \geq h^2 + h + 3.\]
		 	\item If $4 < i_0 < h$, then we have
		 	\[a_{i_0} - a_{i_0 - 1} = a_2.\]
		 	By applying Lemma \ref{rssn-lem16} (See $(1)$ and $(4)$), we get
		 	\[\alpha_{h - i_0} \geq 4.\]
		 	Hence
		 	\[\sum_{n = 0}^{h - 4} \alpha_n = \sum_{\substack{n = 0 \\ n \neq h - i_0}}^{h - 4} \alpha_n + \alpha_{h - i_0} \geq (h - 4) + 4 = h.\]
		 \end{enumerate}
		 
		 \noindent {\textbf{Subcase 2.2.4}} ($t = 0$ ~\text{and}~ $s = 0$).
		 Since $t =0$ and $s = 0$, it follows that 
		 \[a_{i + 1} - a_i = a_2\]
		 for each $i \in [1, h]$. Hence
		 \[A = a_2 \ast [0, h],\] 
and so it follows from Lemma \ref{rssn-lem18} that 
		 \[|h_{\pm}^\wedge A| = |h_{\pm}^\wedge [0, h]| = h^2 + h + 1.\]
	\end{proof}
	
	Combining Lemma \ref{rssn-lem13}, Lemma \ref{rssn-lem14}, Lemma \ref{rssn-lem20}, Lemma \ref{rssn-lem15}, Lemma \ref{rssn-lem18} and Lemma \ref{rssn-lem17}, we get the following Lemma.
	\begin{lemma}\label{rssn-lem23}
		Let $h \geq 4$ be an integer. Let $A = \{a_1, a_2, \ldots, a_{h + 1}\}$ be a set of nonnegative integers such that $0 = a_1 < a_2 < \cdots < a_{h + 1}$. Furthermore, assume that
		\[a_2 \not \equiv 0 \pmod {2}, a_3 = 2a_2.\]
		Then, we have 
		\[|h_{\pm}^\wedge A| \geq h^2 + h + 2,\]
		when either $A \neq a_2 \ast [0, h]$ or $A \neq a_2 \ast \{0, 1, 2, 4, 6\}$. Otherwise, 
		\[|h_{\pm}^\wedge A| = h^2 + h + 1.\]
	\end{lemma}

The following lemma is the part of the result contained in \cite{mmp2024} which will be useful to prove Theorem \ref{rssn-thm-2}.
\begin{lemma}[{\cite[Theorem $6$]{mmp2024}}]\label{thm:10}
Let $h$ and $k$ be positive integers such that $4 \leq h \leq k - 1$. Let $A = \{0, a_2,  \ldots, a_k\}$ be the set of nonnegative integers with $0 < a_2 < \cdots < a_k$ such that
\[|h_{\pm}^\wedge A| = 2hk - h(h + 1) + 1\]
If the set $B = \{0, a_2,  \ldots, a_{h + 1}\} \subseteq A$ is an arithmetic progression, then
\[A = a_2 \ast [0, k - 1].\]
\end{lemma}

\section{Proof of Main Theorems}\label{section-rrsn-thm-proof}	

\begin{proof}[Proof of Theorem \ref{rssn-thm-3}]
Let $A = \{a_1, a_2,\ldots, a_{h + 1}\}$, where $0 = a_1 < a_2 < \cdots < a_{h+1}$. Since 
\[h_{\pm}^\wedge (d \ast A)| = |d \ast h_{\pm}^\wedge A| = |h_{\pm}^\wedge A|\]
for all positive integers $d$, we may assume that 
\[\gcd(a_1, a_2, \ldots, a_{h+1}) = 1.\]
The theorem holds for $h = 3$ as proved in \cite{bhanja-kom-pandey2021}. Therefore, we assume that $h \geq 4$. 
		
\noindent {\textbf{Case 1}} ($a_3 \equiv a_2 \pmod 2$). In this case, if $a_r \not \equiv a_2 \pmod 2$ for some $r \in [4, h + 1]$, then it follows from Lemma \ref{rssn-lem9} that
\[|h_{\pm}^\wedge A| \geq h^2 + h + 2 > h^2 + h + 1.\]
		
If $a_2 \equiv a_3 \equiv \cdots \equiv a_{h + 1} \equiv 0 \pmod$ and $a_2 \not \equiv 0 \pmod 2$, then $A \setminus \{0\} = \{a_2, \ldots, a_{h + 1}\}$ is a set of odd integers. Hence Lemma \ref{rssn-lem10} implies that 
\[|h_{\pm}^\wedge A| \geq 2h(h - 1) - 1 > h^2 + h + 1.\]
		
Finally, if $a_2 \equiv a_3 \equiv \cdots \equiv a_{h + 1} \pmod 2$ and $a_2 \equiv 0 \pmod 2$), then 
\[a_2 \equiv a_3 \equiv \cdots \equiv a_{h + 1} \equiv 0 \pmod 2,\]
and so $\gcd (a_1, a_2, \ldots, a_{h + 1}) \geq 2$, which is a contradiction. Hence this case can not occur.
		
		\noindent {\textbf{Case 2}} ($a_3 \not \equiv a_2 \pmod 2$). Consider the following case.
		
		\noindent {\textbf{Case 2.1}} ($a_3 \not \equiv a_2 \pmod 2$ and $a_2 \equiv 0 \pmod 2$). In this case $a_3 \not \equiv 0 \pmod 2$ and so Lemma \ref{rssn-lem11}, implies that 
		\[|h_{\pm}^\wedge A| \geq h^2 + h + 2 > h^2 + h + 1.\]
		
		\noindent {\textbf{Case 2.2}} ($a_3 \not \equiv a_2 \pmod 2$ and $a_2 \not \equiv 0 \pmod 2$). In this case $a_3 \equiv 0 \pmod 2$. Now if $a_3 \neq 2a_2$, then Lemma \ref{rssn-lem12}, implies that 
		\begin{equation}
			|h_{\pm}^\wedge A| \geq 
			\begin{cases}
				23, &\text{if $h = 4$};\\
				h^2 + h + 2, &\text{if $A_2$ is not in an arithmetic progression and $h \geq 5$};\\
				\frac{3}{2} h(h - 1) + 3, &\text{if $A_2$ is an arithmetic progression and $h \geq 5$.}  
			\end{cases}
		\end{equation} 
		Hence,
		\[|h_{\pm}^\wedge A| \geq h^2 + h + 2 > h^2 + h + 1,\]
		for all $h \geq 4$. 
		
		If $a_3 = 2a_2$, then Lemma \ref{rssn-lem23}, implies that 
		\[|h_{\pm}^\wedge A| \geq h^2 + h + 1.\]
		Therefore, in all cases, we have
		\[|h_{\pm}^\wedge A| \geq h^2 + h + 1.\]

	To show that the lower bound in $\eqref{rssn-thm-3-eq1}$ is best possible, consider the set $A = [0, h]$. Then it follows from Lemma \ref{rssn-lem15} and Lemma \ref{rssn-lem17} that
		\[|h_{\pm}^\wedge A| = h^2 + h + 1.\]		
		This completes the proof. 
	\end{proof}

    \begin{proof}[Proof of Theorem \ref{rssn-thm-1}]  
    	By applying Lemma \ref{rssn-basic-lem1} and Theorem \ref{rssn-thm-3}, we get
    	\[|h_{\pm}^\wedge A| \geq 2hk - h^2 - h + 1.\]
    This establishes the lower bound $\eqref{rssn-thm-1-eq1}$ in Theorem \ref{rssn-thm-1}.
    
    	To show that the lower bound in $\eqref{rssn-thm-1-eq1}$ is best possible, consider the set $A = [0, k - 1]$. Then
    	\[\min (h_{\pm}^\wedge A) = - (k - 1) - \cdots - (k - h) = - hk + \frac{h(h + 1)}{2},\]
    	and 
    	\[\max (h_{\pm}^\wedge A) = (k - 1) + \cdots + (k - h) = hk - \frac{h(h + 1)}{2}.\]
    	Hence
    	\[h_{\pm}^\wedge A \subseteq \biggl[- hk + \frac{h(h + 1)}{2}, hk - \frac{h(h + 1)}{2} \biggl],\]
    	and so
    	\[|h_{\pm}^\wedge A| \leq 2hk - h^2 - h + 1.\]
    	But 
    	\[|h_{\pm}^\wedge A| \geq 2hk - h^2 - h + 1.\]
    	Therefore,
    	\[|h_{\pm}^\wedge A| = 2hk - h^2 - h + 1.\]		
    	This completes the proof. 
    \end{proof}

To prove Theorem \ref{rssn-thm-2}, we need the following lemma also.
\begin{lemma}\label{rssn-lem19}
		Let $h = 4$ and $k \geq 6$ be positive integers. Let $A = \{a_1, a_2, \ldots, a_k\}$ be a set of nonnegative integers such that $0 = a_1 < a_2 < \cdots < a_k$. If
		\[|4_{\pm}^\wedge A| = 8k - 19,\]
	    then $A = a_2 \ast [0, k - 1]$.
\end{lemma}
	    \begin{proof}
	    	First assume that $k = 6$.	In this case 
	    	\[A = \{0, a_2, a_3, a_4, a_5, a_6\}.\]
	    	and 
	    	\[|h_{\pm}^\wedge A| = 29.\]
	    	Let 
	    	\[A_1 = A \setminus \{0\} ~~\text{and}~ A_6 = A \setminus \{a_6\}.\]
	    	Then
	    	\[4^\wedge (- A_1) \cup 4_{\pm}^\wedge A_6 \cup 4^\wedge A_1 \subseteq 4_{\pm}^\wedge A,\]
	    	\[4_{\pm}^\wedge A_6 \cap 4^\wedge A_1 = \{a_2 + a_3 + a_4 + a_5\},\]
	    	and
	    	\[4^\wedge (- A_1) \cap h_{\pm}^\wedge A_6 = \{- a_2 - a_3 - a_4 - a_5\}.\]
	    	Since $|4_{\pm}^\wedge A| = 29$, it follows form Theorem \ref{restricted-hfold-direct-thm} and Theorem Theorem \ref{rssn-thm-1} that
	    	\begin{align*}
	    		29 &= |4_{\pm}^\wedge A|\\
	    		&\geq |4^\wedge (- A_1) \cup 4_{\pm}^\wedge A_6 \cup 4^\wedge A_1|\\
	    		&= |4^\wedge (- A_1)| + |4_{\pm}^\wedge A_6| + |4^\wedge A_1| - 2 \\
	    		&\geq 5 + |4_{\pm}^\wedge A_6| + 5 - 2 \\
	    		&= |4_{\pm}^\wedge A_6| + 8\\
                & \geq 29.
	    	\end{align*}
	    	This implies that
	    	\[|4_{\pm}^\wedge A_6| = 21.\]
	    	Let $A_6 = d \ast A_6'$, where $d = \gcd(a_1, a_2, a_3, a_4, a_5)$ and $A_6' = \{a_1', a_2', a_3', a_4', a_5'\}$. Then $\gcd(a_1', a_2', a_3', a_4', a_5') = 1$.
	    	Since $|4_{\pm}^\wedge A_6'| = |4_{\pm}^\wedge A_6| = 21$, it follows from all the auxiliary lemmas that the equality the equality $|4_{\pm}^\wedge A_6'| = 21$ can occur only in the cases discussed in Lemma \ref{rssn-lem14} and Lemma \ref{rssn-lem15}. These lemmas implies that
	    	\begin{equation*}
	    		~\text{either}~ A_6' = a_2' \ast [0, 4] ~\text{or}~ A_6' = a_2' \ast \{0, 1, 2, 4, 6\}.
	    	\end{equation*}
Since $\gcd(a_1', a_2', a_3', a_4', a_5') = 1$, it follows that $a_2' = 1$, and so $a_2 = da_2' = d$. Therefore,
\begin{equation}\label{rssn-lemeq:20}
	    		~\text{either}~ A_6 = d \ast A_6' = a_2 \ast [0, 4] ~\text{or}~ A_6 = d \ast A_6' = a_2 \ast \{0, 1, 2, 4, 6\}.
	    	\end{equation}
Let $x = 0 + a_3 + a_4 + a_6$. Note that
	    	\[{\max}_{-}(4_{\pm}^\wedge A_6) = 0 + a_3 + a_4 + a_5 < x < {\min}_{+}(4^\wedge A_1)\]
	    	and 
	    	\[{\max}_{-}(4_{\pm}^\wedge A_6) = 0 + a_3 + a_4 + a_5 < a_2 + a_3 + a_4 + a_5 < {\min}_{+}(4^\wedge A_1).\]
	    	Since $|4_{\pm}^\wedge A| = 29$ and $4^\wedge (- A_1) \cup 4_{\pm}^\wedge A_6 \cup 4^\wedge A_1 \subseteq 4_{\pm}^\wedge A$, it follows that
	    	\[x = a_2 + a_3 + a_4 + a_5.\]
	    	Hence
	    	\begin{equation}\label{rssn-lemeq:21}
	    		a_6 = a_5 + a_2
	    	\end{equation}
	    	It follows from \eqref{rssn-lemeq:20} and \eqref{rssn-lemeq:21} that either
	    	\[A = a_2 \ast [0, 5] ~\text{or}~ A = a_2 \ast \{0, 1, 2, 4, 6, 7\}.\]
	    	If $A = a_2 \ast \{0, 1, 2, 4, 6, 7\}$, then
	    	\[|4_{\pm}^\wedge A| = |4_{\pm}^\wedge \{0, 1, 2, 4, 6, 7\}| \geq 30,\]
	    	which is a contradiction. Therefore,
	    	\[A = a_2 \ast [0, 5].\]

	    	Now assume that $k \geq 7$.	Let 
	    	\[B = \{0, a_2, a_3, a_4, a_5\}\]
	    	and 
	    	\[A_1 = \{a_2, \ldots, a_k\}.\]
	    	Similar to the argument for the case $k = 6$, we can show that
	    	\[B = a_2 \ast [0, 4] ~\text{or}~ B = a_2 \ast \{0, 1, 2, 4, 6\},\]
	    	and
	    	\[|4^\wedge A_1| = 4k - 19.\]
	    	Since $k \geq 7$ and $|4^\wedge A_1| = 4k - 19$, it follows Theorem \ref{restricted-hfold-inverse-thm} that $A_1$ is an arithmetic progression. Since $A_1$ is an arithmetic progression, it follows that the only possibility for the set $B$ is $B = a_2 \ast [0, 4]$. Hence
	    	\[A = a_2 \ast [0, k - 1].\]
	    	This completes the proof.
	    \end{proof}
		
\begin{proof}[Proof of Theorem \ref{rssn-thm-2}]
The theorem holds for $h = 3$ as proved in \cite{bhanja-kom-pandey2021}. Therefore, we assume that $h \geq 4$. Let $A = \{a_1, a_2,\ldots, a_k\}$, where $0 = a_1 < a_2 < \cdots < a_k$. Since 
\[h_{\pm}^\wedge (d \ast A)| = |d \ast h_{\pm}^\wedge A| = |h_{\pm}^\wedge A|\]
for all positive integers $d$, we may assume that 
\[\gcd(a_1, a_2, \ldots, a_k) = 1.\]
If $h = 4$ and $k = 5$, then let $A = \{a_1, a_2, a_3, a_4, a_5\}$, where $0 = a_1 < a_2 < a_3 < a_4 < a_5$. Let $A = d \ast A'$, where $A' = \{a_1', a_2', a_3', a_4', a_5'\}$ and $\gcd(a_1', a_2', \ldots, a_5') = 1$. It follows from all the auxiliary lemmas that the equality the equality $|4_{\pm}^\wedge A'| = |4_{\pm}^\wedge A| = 21$ can occur only in the case discussed in Lemma \ref{rssn-lem14} and  Lemma \ref{rssn-lem15}. Lemma \ref{rssn-lem14} and  Lemma \ref{rssn-lem15} implies that
\[\text{either}~A' = a_2' \ast [0, 4] ~\text{or}~ A' = a_2' \ast \{0, 1, 2, 4, 6\}.\]
Since $\gcd(a_1', a_2', a_3', a_4', a_5') = 1$, it follows that $a_2' = 1$, and so $a_2 = da_2' = d$. Therefore,
\[\text{either}~A = d \ast A' = a_2 \ast [0, 4] ~\text{or}~ A = d \ast A' = a_2 \ast \{0, 1, 2, 4, 6\}.\]

Similarly if $h = 4$ and $k \geq 6$, then let $A = \{a_1, a_2, \ldots, a_k\}$, where $0 = a_1 < a_2 < \cdots < a_k$. Let $A = d \ast A'$, where $A' = \{a_1', a_2', \ldots, a_k'\}$ and $\gcd(a_1', a_2', \ldots, a_k') = 1$. It follows from all the auxiliary lemmas that the equality the equality $|4_{\pm}^\wedge A'| = |4_{\pm}^\wedge A| = 8k - 19$ can occur only in the case discussed in Lemma \ref{rssn-lem19}. This lemma implies that
\[A' = a_2' \ast [0, k - 1].\]
Since $\gcd(a_1', a_2', \ldots, a_k') = 1$, it follows that $a_2' = 1$, and so $a_2 = da_2' = d$. Therefore,
\[A = d \ast A' = a_2 \ast [0, k - 1].\]

Now assume that $h \geq 5$. Let $A = \{a_1, \ldots, a_k\}$, where $0 = a_1 < \cdots < a_k$. Let $B = \{a_1, \ldots, a_{h + 1}\} \subseteq A$, and let $A' = A \setminus \{a_1\}$. Then
		\[(-h^{\wedge} A') \cup h_{\pm}^{\wedge} B \cup h^{\wedge} A' \subseteq  h_{\pm}^{\wedge} B.\]
		Since
		\[(-h^{\wedge} A') \cap h_{\pm}^{\wedge} B = \{-a_2 - \cdots - a_{h + 1}\},\]
		and
		\[h_{\pm}^{\wedge} B \cap h^{\wedge} A' = \{a_2 + \cdots + a_{h + 1}\},\]
		it follows that
		\[|h_{\pm}^{\wedge} A| \geq |h_{\pm}^{\wedge} B| + 2 |h^{\wedge} A'| - 2.\]
		Therefore, applying Theorem \ref{restricted-hfold-direct-thm} and Theorem \ref{rssn-thm-3}, we get
		\begin{align*}
			2hk - h^2 - h + 1 = |h_{\pm}^{\wedge} A| & \geq |h_{\pm}^{\wedge} B| + 2 |h^{\wedge} A'| - 2\\
			& \geq |h_{\pm}^{\wedge} B| + 2(h(k - 1) - h^2 + 1) - 2 \\
			& \geq (h^2 + h + 1) + 2(h(k - 1) - h^2 + 1) - 2\\
			& = 2hk - h^2 - h + 1.
		\end{align*}
The above inequalities implies that
\[|h_{\pm}^{\wedge} B| + 2(h(k - 1) - h^2 + 1) - 2  = 2hk - h^2 - h + 1,\]
and so
\[|h_{\pm}^{\wedge} B| = h^2 + h + 1.\]
Now let $d = \gcd (a_1, \ldots, a_{h + 1})$, and let $a_i' = a_i/d$ for $i \in [1, h + 1]$. Then $\gcd (a_1', \ldots, a_{h + 1}') = 1$. Let $B' = \{a_1', \ldots, a_{h + 1}'\}$. Then $B = d \ast B'$, and so
\begin{equation}\label{rssn-thm-2-eq2}
|h_{\pm}^{\wedge} B'| = |h_{\pm}^{\wedge} B| = h^2 + h + 1.
\end{equation}
It follows from all the auxiliary lemmas that the equality in $\eqref{rssn-thm-2-eq2}$ can occur only in the case discussed in Lemma \ref{rssn-lem18} and Lemma \ref{rssn-lem17}. Then Lemma \ref{rssn-lem18} and Lemma \ref{rssn-lem17} implies that $B' = [0, h]$. Thus $d = a_2/ a_2' = a_2$, and so
\[B = d \ast B' = a_2 \ast [0, h].\]
Now since $B$ is an arithmetic progression and $|h_{\pm}^\wedge A| = 2hk - h(h + 1) + 1$, it follows from Lemma \ref{thm:10} that 
\[A = a_2 \ast [0, k - 1].\]
This completes the proof.
\end{proof}

	

\end{document}